\documentclass[12pt,spanish,english]{amsart}
\usepackage[T1]{fontenc}
\usepackage[latin9]{inputenc}
\usepackage{geometry}
\usepackage{amsthm}
\usepackage{amsmath}
\usepackage{amssymb}
\usepackage{esint}

\usepackage[numbers]{natbib}
\usepackage[colorlinks=true, pdfstartview=FitV, linkcolor=blue, citecolor=blue, urlcolor=blue]{hyperref}
\usepackage{doi}
\usepackage{color}

\makeatletter
\theoremstyle{plain}
  \newtheorem{thm}{\protect\theoremname}
  \theoremstyle{plain}
  \newtheorem*{thm*}{\protect\theoremname}
  \theoremstyle{plain}
  \newtheorem{cor}{\protect\corollaryname}
  \theoremstyle{plain}
  \newtheorem*{lem*}{\protect\lemmaname}
  \theoremstyle{plain}
  \newtheorem{lem}{\protect\lemmaname}

\DeclareMathOperator{\supp}{supp}

\makeatother

\usepackage{babel}
\addto\shorthandsspanish{\spanishdeactivate{~<>}}

  \addto\captionsenglish{\renewcommand{\lemmaname}{Lemma}}
  \addto\captionsspanish{\renewcommand{\lemmaname}{Lema}}
  \providecommand{\lemmaname}{Lemma}
\addto\captionsenglish{\renewcommand{\corollaryname}{Corollary}}
\addto\captionsenglish{\renewcommand{\theoremname}{Theorem}}
\addto\captionsspanish{\renewcommand{\corollaryname}{Corolario}}
\addto\captionsspanish{\renewcommand{\theoremname}{Teorema}}
\providecommand{\corollaryname}{Corollary}
\providecommand{\theoremname}{Theorem}

\begin{document}

\title{ Borderline weighted estimates for commutators  of singular integrals
} 

\author{Carlos P\'erez}
\address{Carlos P\'erez, Department of Mathematics, University of the Basque Country UPV/EHU and
IKERBASQUE, Basque Foundation for Science, Bilbao, Spain.
}
\email{carlos.perezmo@ehu.es}

\author{Israel P. Rivera-R\'{\i}os}
\address{Israel P. Rivera-R\'{\i}os, IMUS \& Departamento de Análisis Matemático, Universidad de Sevilla, Sevilla, Spain}
\email{petnapet@gmail.com}

 \subjclass[2010]{42B35,46E30}

 \keywords{Commutators, Rubio de Francia Extrapolation; $A_p$ weights; Hardy-Littlewood maximal function}

\thanks{The first author was  supported by Grant MTM2014-53850-P and the second author was supported by Grant MTM2012-30748, Spanish Government}

 \begin{abstract} 
In this paper we establish the following estimate 
\[
w\left(\left\{ x\in\mathbb{R}^{n}\,:\,\left|[b,T]f(x)\right| > \lambda\right\} \right)\leq \frac{c_{T}}{\varepsilon^{2}}\int_{\mathbb{R}^{n}}\Phi\left(\|b\|_{BMO}\frac{|f(x)|}{\lambda}\right)M_{L(\log L)^{1+\varepsilon}}w(x)dx
\]
where $w\geq0, \, 0<\varepsilon<1$ and $\Phi(t)=t(1+\log^+(t))$.
This inequality relies upon the following sharp $L^p$ estimate
\[
\|[b,T]f\|_{L^{p}(w)}\leq c_{T}\left(p'\right)^{2}p^{2}\left(\frac{p-1}{\delta}\right)^{\frac{1}{p'}} \|b\|_{BMO} \, \|f \|_{L^{p}(M_{L(\log L)^{2p-1+\delta}}w)}
\]where  $1<p<\infty, w\geq0 \text{ and } 0<\delta<1.$
As a consequence we recover the following estimate essentially contained in \cite{MR3008263}:
\[w\left(\{x\in\mathbb{R}^{n}\,:\,\left|[b,T]f(x)\right| >\lambda\}\right)\leq c_T\,[w]_{A_{\infty}}\left(1+\log^{+}[w]_{A_{\infty}}\right)^{2}\int_{\mathbb{R}^{n}} \Phi\left(\|b\|_{BMO}\frac{|f(x)|}{\lambda}\right)Mw(x)dx\]

We also obtain the analogue estimates for symbol-multilinear commutators for a wider class of symbols.

\end{abstract}
\maketitle

\section{Introduction}

Motivated by a classical inequality due to C. Fefferman and E. Stein
for the Hardy-Littlewood maximal function, namely 
\[
\|Mf\|_{L^{1,\infty}(w)}\leq c\int_{\mathbb{R}^{n}}|f|Mwdx
\]
where $M$ denotes the Hardy-Littlewood maximal operator and $w$
is a weight, i.e. a locally non negative integrable function, B. Muckenhoupt
and R. Wheeden conjectured that 
\[
\|Hf\|_{L^{1,\infty}(w)}\leq c\,\int_{\mathbb{R}^{n}}|f|Mwdx
\]
where $H$ is the Hilbert transform. This conjecture was recently
disproved by M. C. Reguera and C. Thiele \cite{MR2923171} (see also \cite{2013arXiv1312.5255C} for the result in higher dimensions). 
The failure of this conjecture was suggested by the first author in \cite{MR1260114} where the following positive result was obtained
\begin{equation}\label{EqMax}
\|Tf\|_{L^{1,\infty}(w)}\leq c_{\varepsilon,T}\int_{\mathbb{R}^{n}}|f|M_{L\left(\log L\right)^{\varepsilon}}(w)dx\qquad w\geq0,
\end{equation}
where $T$ is a Calderón-Zygmund operator (CZO).  In the recent work \cite{MR3327006}, T. Hytönen and the first 
author improved the control on the $c_{\varepsilon,T}$ constant and were 
able to consider the maximal singular operator 
$T^{*}$ obtaining the following estimate (see \cite{2015arXiv150501804D} for an improvement of this result)
\begin{equation}\label{1/epsestimate}
\|T^{*}f\|_{L^{1,\infty}(w)}\leq\frac{c_{T}}{\varepsilon}\int_{\mathbb{R}^{n}}|f(x)|M_{L\left(\log L\right)^{\varepsilon}}(w)(x)dx\qquad w\geq0
\end{equation}
which implies  
\begin{equation}\label{lognec}
\|T^{*}f\|_{L^{1,\infty}(w)}\leq c_T\, \log\left(e+[w]_{A_{\infty}}\right)\int_{\mathbb{R}^{n}}|f|Mwdx,
\end{equation}
when $w\in A_{\infty}$. This result improves the main theorem from \cite{MR2480568}, namely 
\begin{equation}\label{A1conj}
\|T^*\|_{ L^1(w)  \to L^{1,\infty}(w)  } \leq c_T\,[w]_{A_1}  \log\left(e+[w]_{A_{\infty}}\right). 
\end{equation}
It seemed that the logarithmic  factor was superfluous  and that it could be removed. However, this is not the case by a very impressive negative result obtained by F. Nazarov, A. Reznikov, V. Vasyunin and A. Volberg in \cite{2015arXiv150604710N}. In this work the authors disprove the so called $A_1$ conjecture, namely they prove
\begin{equation*}
\sup_{ w\in A_{1} } \frac{ \|H\|_{ L^1(w)  \to L^{1,\infty}(w)  } }{[w]_{A_1}} =\infty
\end{equation*}
where $H$ is the classical Hilbert transform. Furthermore, the same conclusion holds if the linear constant $[w]_{A_1}$  is replaced by $[w]_{A_1} \log(e+[w]_{A_1} )^{\alpha}$ for a  positive $\alpha<\frac{1}{5}$. This is indicating that most probably  \eqref{A1conj} is fully optimal.

The main purpose of this paper is to prove estimates similar to
\eqref{1/epsestimate} for commutators of CZOs $T$ with $BMO$ functions $b$, usually called the symbol. These operators are defined formally by the expression
$$[b,T]f=b T(f) - T(b\,f), $$
These commutators were introduced  by Coifman, Rochberg and Weiss in \cite{MR0412721} in connection with the classical factorization theorem for Hardy spaces. However,  many other applications were found much later, specially in the theory of elliptic operators \cite{MR1631658}, \cite{MR1088476}. Another interesting aspect of the theory is its connection with the following nonlinear commutators introduced by R. Rochberg and G. Weiss in \cite{MR717826}: 
\[
f
\rightarrow
Nf= T(f\,\log |f| )  - Tf\, \log |Tf|.
\]
This operator is interesting due to its relationship with the Jacobian mapping and with nonlinear P.D.E.
as shown in \cite{MR1288682} and \cite{MR1259100}.

The main result from \cite{MR0412721} states that $[b, T]$ is a bounded
operator on $L^{p}( \mathbb R^{n} )$, $1<p<\infty$, when $b$ is a
$BMO$ function and $T$ is a singular integral operator. In fact, the
$BMO$ condition of $b$ is also a necessary condition for the
$L^{p}$-boundedness of the commutator when $T$ is the Hilbert
transform.


From the theoretical point of view, these commutators are of interest because they are more singular than CZOs.  For instance, the first author proved in  \cite{MR1317714}  that these commutators are
not of weak type $(1,1)$ obtaining a suitable replacement, namely
the following $L\log L$ endpoint estimate: 
\[
w\left(\left\{ x\in\mathbb{R}^{n}\,:\,\left|\left[b,T\right]f(x)\right|>\lambda\right\} \right)\leq c\,\int_{\mathbb{R}^{n}}\Phi\left(\frac{|f|}{\lambda}\|b\|_{BMO}\right)wdx
\]
where $\Phi(t)=t\left(1+\log^{+}(t)\right)$,  $w\in A_{1}$ and the constant $c$ depends upon the $A_{1}$ constant of the weight.  The approach to prove this result was based on an appropriate non-standard good-$\lambda$ inequality using the 
 Fefferman-Stein   ``sharp'' maximal function. However, this method does not produce good results for
further developments and in particular when considering  non-$A_{\infty}$ weights. 

Later on, the first author together with G. Pradolini (\cite{MR1827073}) established the following estimate for arbitrary weights $w\geq 0$,
\[
w\left(\left\{ x\in\mathbb{R}^{n}\,:\,\left|\left[b,T\right]f(x)\right|>\lambda\right\} \right)\leq c_{T,\varepsilon}\int_{\mathbb{R}^{n}} \Phi\left(\frac{|f|}{\lambda}\|b\|_{BMO}\right)M_{L\left(\log L\right)^{1+\varepsilon}}wdx
\]
for any $\varepsilon>0$. The aim of this paper is to give a quantitative version of this estimate with a good control on the constant $C_{T,\varepsilon}$ in terms of $\varepsilon$.  In the simplest situation our estimate can be stated as follows 
\begin{equation}\label{eq:EndpIntro}
w\left(\left\{ x\in\mathbb{R}^{n}\,:\,\left|\left[b,T\right]f(x)\right|>\lambda\right\} \right)\leq \frac{c_{T}}{\varepsilon^{2}}\int_{\mathbb{R}^{n}} \Phi\left(\frac{|f|}{\lambda}\|b\|_{BMO}\right)M_{L\left(\log L\right)^{1+\varepsilon}}wdx.
\end{equation}

See Theorem  \ref{Extremo} for the general situation. In fact, we will obtain a wider class of  results since we will be considering symbol-multilinear commutators with symbols in $Osc_{\exp L^{s}}$ classes which are subspaces of the BMO space (cf. Section \ref{sub:Orlicz-maximal-functions}). We will see that this choice of symbols will be reflected in the maximal operator on the right hand side of the inequality.  As a consequence of these type of estimates we can recover, among other results, the following endpoint $A_{1}$ result from C. Ortiz \cite{MR3008263} (see Corollary \ref{Corolario1}):
\[
w\left(\left\{ x\in\mathbb{R}^{n}\,:\,\left|\left[b,T\right]f(x)\right|>\lambda\right\} \right)\leq c_{T}\,\Phi\left([w]_{A_{1}}\right)^{2}\int_{\mathbb{R}^{n}} \Phi\left(\|b\|_{BMO}\frac{|f(x)|}{\lambda}\right)w(x)dx.
\]

Estimate \eqref{eq:EndpIntro} should be compared with the case of  CZOs  \eqref{1/epsestimate}.  It is not clear whether is possible or not to establish \eqref{eq:EndpIntro} using techniques based on sparse operators as in \cite{2015arXiv150501804D}. Another open question is the analogue of the Muckenhoupt-Wheeden conjecture for the commutator, namely whether
$$w\left(\left\{ x\in\mathbb{R}^{n}\,:\,\left|\left[b,T\right]f(x)\right|>\lambda\right\} \right)\leq c_w\int_{\mathbb{R}^{n}} \Phi\left(\|b\|_{BMO}\frac{|f(x)|}{\lambda}\right)M^2(w)(x)dx.$$
holds for every weight $w$ or not. Techniques used in \cite{MR2923171} and \cite{2013arXiv1312.5255C} rely upon an endpoint extrapolation result, firstly established in \cite{MR1761362}, or upon variations of it as in    \cite{2015arXiv150501804D}. It is not clear how to perform a similar extrapolation from the $L\log L$ estimate that commutators satisfy.

This paper is organized as follows: Section \ref{MainResults} contains the statements of our main results and  the proof of Corollary \ref{Corolario1}. Section \ref{PN} contains precise definitions and facts which will be used throughout the paper. In section \ref{PTF} we give the proof of the strong type theorem, namely Theorem \ref{TipoFuerte}, and all the needed technical results. The proof of the endpoint estimate (Theorem \ref{Extremo}) is presented in Section \ref{PTD}.

\section*{Acknowledgments}

The authors are very grateful to Carmen Ortiz-Caraballo for her interest in this work and for suggesting  
some improvement in the presentation of the paper.  We also wish
to thank the referee for valuable comments on the paper.

\section{Main results}\label{MainResults}

To state our main results we need to introduce some notation.  Let $b_{i}\in Osc_{\exp L^{s_i}}$
$s_i\geq 1$, $i=1,\cdots,k$ (cf. section \ref{sub:Orlicz-maximal-functions} after Lemma \ref{Lemma:HolderGeneralizado}) and $T$ a CZO with associated kernel $K$. We define the symbol-multilinear
commutator with respect to the symbol $\vec{b}=\left(b_{1},\dots,b_{k}\right)$
as follows 
\[
T_{\vec{b}}f(x)=\int_{\mathbb{R}^{n}} \prod_{i=1}^{k}\left(b_{i}(x)-b_{i}(y)\right)K(x,y)f(y)dy
\]
We also denote
\[
\frac{1}{s}=\sum_{i=1}^{k}\frac{1}{s_{i}}
\]
and 
\[
\|\vec{b}\|=\prod_{i=1}^{k}\|b_{i}\|_{Osc_{\exp L^{s_{i}}}}.
\]
Our main results are the following. 

\begin{thm}\label{TipoFuerte} Let be \,$T_{\vec{b}}$ defined as above and let $w$
be a weight. Then
\[
\|T_{\vec{b}}f\|_{L^{p}(w)}\leq c_{T}\left(p'\right)^{k+1}p^{1+\frac{1}{s}}\left(\frac{p-1}{\delta}\right)^{\frac{1}{p'}} \|\vec{b} \| \, \|f \|_{L^{p}(M_{L(\log L)^{(1+\frac{1}{s})p-1+\delta}}w)}
\]
for every $\delta\in\left(0,1\right)$ and $p\in\left(1,\infty\right)$. 
\end{thm}


This result can be applied to derive the following endpoint estimate.

\begin{thm}
\label{Extremo}Let be $T_{\vec{b}}$  and $w$ as above.  Then 
\[
w\left(\left\{ x\in\mathbb{R}^{n}\,:\,\left|T_{\vec{b}}f\right| > \lambda\right\} \right)\leq \frac{c_{T}}{\varepsilon^{k+1}}\int_{\mathbb{R}^{n}}\Phi_{\frac{1}{s}}\left(\|\vec{b}\|\frac{|f(x)|}{\lambda}\right)M_{L(\log L)^{\frac{1}{s}+\varepsilon}}w(x)dx
\]
for every $\varepsilon\in(0,1)$ where $\Phi_{\rho}(t)=t(1+\log^{+}(t))^{\rho},$ $\rho>0.$
\end{thm}

There is an interesting application of Theorem \ref{Extremo} from which 
we can recover one of the main results of \cite{MR3008263}.

\begin{cor}
\label{Corolario1}Let $T$ be a CZO and let $b\in BMO$.
\begin{enumerate}
\item If $w\in A_{\infty}$ then 
\[
w\left(\{x\in\mathbb{R}^{n}\,:\,\left|[b,T]f(x)\right| >\lambda\}\right)\leq c[w]_{A_{\infty}}\left(1+\log^{+}[w]_{A_{\infty}}\right)^{2}\int_{\mathbb{R}^{n}} \Phi\left(\|b\|_{BMO}\frac{|f(x)|}{\lambda}\right)Mw(x)dx, 
\]
\item If $w\in A_{1}$ then \textup{
\[
\begin{split} & w\left(\{x\in\mathbb{R}^{n}\,:\,\left|[b,T]f(x)\right| > \lambda\}\right)\\
 & \leq c[w]_{A_{1}}[w]_{A_{\infty}}\left(1+\log^{+}[w]_{A_{\infty}}\right)^{2}\int_{\mathbb{R}^{n}} \Phi\left(\|b\|_{BMO}\frac{|f(x)|}{\lambda}\right)w(x)dx\\
 & \leq c\Phi\left([w]_{A_{1}}\right)^{2}\int_{\mathbb{R}^{n}} \Phi\left(\|b\|_{BMO}\frac{|f(x)|}{\lambda}\right)w(x)dx
\end{split}
\]
}
\end{enumerate}
where $\Phi(t)=t(1+\log^{+}t)$.

\end{cor}

\begin{proof}

For the proof of the corollary we follow the arguments in \cite{MR3327006}. First observe that for every $\alpha>0$ we have that $\log t\leq\frac{t^{\alpha}}{\alpha}$.
Then we can write
\[
\log(t)^{1+\varepsilon}\leq\frac{t^{\alpha(1+\varepsilon)}}{\alpha^{1+\varepsilon}},
\]
and hence 
\[
M_{L(\log L)^{1+\varepsilon}}w\leq\frac{1}{\alpha^{1+\varepsilon}}M_{L^{1+\alpha(1+\varepsilon)}}.
\]

Let us take $\alpha=\frac{1}{c_{n}[w]_{A_{\infty}}(1+\varepsilon)}$.
Then, using the reverse Hölder inequality (Theorem \ref{thm:SharpRHI}),
\[
\begin{split}\frac{1}{\varepsilon^{2}}M_{L\left(\log L\right)^{1+\varepsilon}} & \leq\frac{1}{\varepsilon^{2}}\left[c_{n}[w]_{A_{\infty}}(1+\varepsilon)\right]^{1+\varepsilon}M_{L^{1+\alpha(1+\varepsilon)}}w\\
 & \leq\frac{2}{\varepsilon^{2}}\left[c_{n}[w]_{A_{\infty}}(1+\varepsilon)\right]^{1+\varepsilon}Mw.
\end{split}
\]
If we choose $\varepsilon=\frac{1}{1+\log^{+}\left([w]_{A_{\infty}}\right)}$
we obtain the desired results just recalling that $[w]_{A_{\infty}}\leq[w]_{A_{1}}$.

\end{proof}

\section{Preliminaries and notation} \label{PN}
In this section we gather some definitions and properties which will be used throughout the paper. 

\subsection{$A_p$ weights}

We recall that a weight $w$ belongs to the class $A_p$, $1< p < \infty$, if
$$
[w]_{A_{p}}=\sup_{Q}\left(\frac{1}{|Q|}\int_{Q}w\right)\left(\frac{1}{|Q|}\int_{Q}w^{-\frac{1}{p-1}}\right)^{p-1}<\infty.$$
A weight $w$ belongs to the class $A_1$ if there is a finite constant $C$ such that
$$ \frac{1}{ |Q| } \int_{Q} w(y)\, dy \le C\inf_{Q} w, \label{A11}
$$
and the infimum of these constants $C$ is called the $A_1$ constant
of $w$ denoted by $[w]_{A_{1}}$. Since the $A_p$ classes are increasing with respect to $p$,
the $A_{\infty}$ class of weights is defined in a natural way by
$A_{\infty} = \cup_{p>1} A_p$. These classes of weights were introduced by B. Muckenhoupt in \cite{MR0293384} where it was shown that for \,$1<p<\infty$
 \[
 w\in A_{p}\, \iff M:L^{p}(w)\longrightarrow L^{p}(w)     
 \] 
and also
 \[ w\in A_{1}\iff M:L^{1}(w)\longrightarrow L^{1,\infty}(w).\] 
From the definition of $A_{\infty}$ it is not clear how to define an appropriate constant. However, Fujii proved essentially in \cite{MR0481968} another characterization: 
$$w\in A_{\infty}\iff[w]_{A_{\infty}}=\sup_{Q}\frac{1}{w(Q)}\int_{Q}M(\chi_{Q}w)dx<\infty$$
which was also rediscovered later on by Wilson in \cite{MR883661}. Recently, this quantity was defined as the $A_\infty$ constant in \cite{MR3092729} since it was proved to be the most suitable one. 
In particular, the following  optimal reverse H\"older's inequality obtained in \cite{MR3092729} (see also \cite{MR2990061} for a better proof and \cite{MR1979718} for some other related results) was used in the proof of Corollary \ref{Corolario1}.

\begin{thm}  \label{thm:SharpRHI}

Let \,$w \in A_{\infty}$, then there exists a dimensional constant $\tau_{n}$ such that
\begin{equation*}
  \left(\frac{1}{|Q|}\int_{Q}w^{r_{w}}\right)^{\frac{1}{r_{w}}}\leq\frac{2}{|Q|}\int_{Q}w.
\end{equation*}
where 
$$r_w=1+\frac{1}{\tau_{n} [w]_{A_{\infty} }}$$
\end{thm}

\subsection{Orlicz maximal functions\label{sub:Orlicz-maximal-functions}}

We recall that $\Phi$ is a Young function if it
is a continuous, nonnegative, strictly increasing and convex function
defined on $[0,\infty)$ such that $\Phi(0)=0$ and $\lim_{t\rightarrow\infty}\Phi(t)=\infty$. We define the localized Luxembourg norm of a function $f$
with respect to a Young function $\Phi$ as follows 
\[
\|f\|_{\Phi,Q}=\|f\|_{\Phi(L),Q}= \inf\left\{ \lambda>0\,:\,\frac{1}{|Q|}\int_{Q}\Phi\left(\frac{|f(x)|}{\lambda}\right)dx\leq1\right\}
\]
which is equivalent to the following
\[
\|f\|'_{\Phi,Q}=\inf_{\mu>0}\left\{ \mu+\frac{\mu}{|Q_{j}|}\int_{Q_{j}}\Phi\left(\frac{|f(x)|}{\mu}\right)dx\right\}. 
\]
This result is due to Krasnosel'ski\u\i, M. A. and Ruticki\u\i, Ja. B. \cite[p. 92]{MR0126722} (see also \cite[p. 69]{Rao1991}). In fact,
\[\|f\|_{\Phi,Q}\leq \|f\|'_{\Phi,Q}\leq 2 \|f\|_{\Phi,Q}\]
which will be quite useful for our purposes.  Observe that the case $\Phi(t)=t$ corresponds to the usual average and we can see these localized norms as a ``different'' way of taking averages. We can also define the maximal function associated to $\Phi$ as
\[
M_{\Phi}f(x)=\sup_{x\in Q}\|f\|_{\Phi,Q}.
\]
Some useful examples that will be quite useful in the sequel are $L\log L$ functions 
\[ \Phi_{\rho}(t)=t(1+\log^{+}(t))^{\rho} \quad \text{with} \quad t \geq 0\]
where $\log^{+}(t)=\chi_{(1,\infty)}(t)\log(t)$ and $\rho>0$. For such $\Phi$
we shall denote \[\|f\|_{\Phi,Q}=\|f\|_{L\left(\log L\right)^{\rho},Q}.\]

Another useful property that makes interesting these ``non-standard
averages'' is the following generalized Hölder inequality.

\begin{lem}
\label{Lemma:HolderGeneralizado}Let $\Phi_{0},\Phi_{1},\Phi_{2},\dots,\Phi_{k}$
be Young functions. If 
\begin{equation}
\Phi_{1}^{-1}(t)\Phi_{2}^{-1}(t)\dots\Phi_{k}^{-1}(t)\leq\kappa\Phi_{0}^{-1}(t).\label{eqHgen}
\end{equation}
then for all functions $f_{1},\dots,f_{m}$ and all cubes $Q$ we
have that 
\[
\|f_{1}f_{2}\dots f_{k}\|_{\Phi_{0},Q}\leq k\kappa\|f_{1}\|_{\Phi_{1},Q}\|f_{2}\|_{\Phi_{2},Q}\dots\|f_{k}\|_{\Phi_{k},Q}.
\]

\end{lem}

A particular case of interest, an especially in this paper, are the spaces defined 
by 
 \[
\|f\|_{Osc_{expL^{s}}}=\sup_{Q}\|f-f_{Q}\|_{\Psi_{s},Q}
\]
where 
\[
\Psi_{s}(t)=e^{t^{s}}-1 \qquad t\geq 0,
\]
with $s>0$, is a Young function.  Then the space $Osc_{\exp L^{s}}$ is defined as 
\[
Osc_{\exp L^{s}}=\left\{ f\in L_{loc}^{1}(\mathbb{R}^{n})\,:\,\|f\|_{Osc_{expL^{s}}}<\infty\right\} .
\]
We observe that John-Nirenberg's theorem yields $BMO=Osc_{\exp L}$. 
It's also clear that for every $s>1$ 
\[
Osc_{\exp L^{s}}\subsetneq BMO.
\]

Now we state a result borrowed from \cite{MR1895740} that will be used in the proof of Theorem  \ref{Extremo}.
\begin{lem} \label{Lema2.2Pt}Let $\Phi_{0},\dots,\Phi_{k}$
be continuous, nonnegative, strictly increasing functions on $[0,\infty)$
with $\Phi(0)=0$ and $\lim_{t\rightarrow\infty}\Phi(t)=\infty$ such
that 
\[
\Phi_{1}^{-1}(t)\Phi_{2}^{-1}(t)\dots\Phi_{k}^{-1}(t)\leq\Phi_{0}^{-1}(t)   \qquad t\geq 0, 
\]
then for all $0\leq x_{1},x_{2},\dots,x_{k}<\infty$ 
\[
\Phi_{0}(x_{1}x_{2}\dots x_{k})\leq\Phi_{1}(x_{1})+\Phi_{2}(x_{2})+\dots+\Phi_{k}(x_{k}).
\]

\end{lem}
To close this section we provide a proof of Lemma \ref{Lemma:HolderGeneralizado} and also a corollary of it that will be quite useful in the proof of Theorem \ref{Extremo}.
\begin{proof} Fix $(x_1,\dots, x_k)$ and consider $t_{0}=\Phi_{1}(x_{1})+\Phi_{2}(x_{2})+\dots+\Phi_{k}(x_{k})$. Combining  \eqref{eqHgen} and the fact that each  $\Phi_{i}$ is 
increasing it readily follows that 
\[
\Phi_{0}\left(\frac{\Phi_{1}^{-1}(t_{0})\Phi_{2}^{-1}(t_{0})\dots\Phi_{k}^{-1}(t_{0})}{\kappa}\right)\leq t_{0}
\]
and also that 
\[
\Phi_{i}^{-1}(t_{0})\geq\Phi_{i}^{-1}(\Phi_{i}(x_{i}))=x_{i}.
\]
Then we have that 
\begin{equation}
\Phi_{0}\left(\frac{x_{1}x_{2}\dots x_{k}}{\kappa}\right)\leq\Phi_{1}(x_{1})+\Phi_{2}(x_{2})+\dots+\Phi_{k}(x_{k})\label{eq:HGen2}
\end{equation}
We observe that this argument gives us a proof of Lemma \ref{Lema2.2Pt}.
Coming back to our proof, let us consider now $t_{i}>\|f_{i}\|_{\Phi_{i},Q}$.
We have that using \ref{eq:HGen2}, 
\[
\begin{split} & \frac{1}{m}\frac{1}{|Q|}\int_{Q}\Phi_{0}\left(\frac{|f_{1}\dots f_{k}|}{\kappa t_{1}\dots t_{k}}\right)\\
 & \leq\frac{1}{m}\left(\frac{1}{|Q|}\int_{Q}\Phi_{1}\left(\frac{|f_{1}|}{t_{1}}\right)+\dots+\frac{1}{|Q|}\int_{Q}\Phi_{k}\left(\frac{|f_{k}|}{t_{k}}\right)\right)\\
 & <1
\end{split}
\]
Consequently 
\[
\|f_{1}\dots f_{k}\|_{\Phi_{0},Q}\leq\kappa t_{1}\dots t_{k}
\]
and it is enough to take the infimum on each $t_{i}$ to finish the proof of the lemma. 
\end{proof}

As a particular case, the following corollary holds which will be used several times in this paper.

\begin{cor}\label{CorolarioInv}
Let $s_1,\dots,s_k\geq 1$ and denote $\sum^{k}_{i=1}\frac{1}{s_i}$. Then
\[\frac{1}{|Q|}\int_Q |f_1\dots f_k g| \leq c_s \|f_1\|_{\exp L^{s_1},Q}\dots\|f_k\|_{\exp L^s_k,Q}\|g\|_{L(\log L)^\frac{1}{s},Q} \]
\end{cor}
\begin{proof}
We denote $\varphi_{\eta}(t)=e^{t^{\eta}}-1$. Then $\varphi_{\eta}^{-1}(t)=\log(x+1)^{\frac{1}{\eta}}$ and we have
that 
\[\varphi_{s_{1}}^{-1}(t)\dots\varphi_{s_{k}}^{-1}(t)\Phi^{-1}_\frac{1}{s}(t)\simeq
\varphi_{s_{1}}^{-1}(t)\dots\varphi_{s_{k}}^{-1}(t)\frac{x}{\log\left(x+1\right)^{\frac{1}{s}}}\leq x
\]
and Lemma \ref{Lemma:HolderGeneralizado} gives the desired result.
\end{proof}

\subsection{Symbol-multilinear commutators}

We recall that an operator $T$ initially defined on the Schwartz spaces and taking values into the space of
tempered distributions\, $T:S({\mathbb R}^n) \to S'({\mathbb R}^n)$\, is a CZO if, 

\begin{enumerate}
\item $T$ is bounded on $L^{2}(\mathbb{R}^{n})$. 
\item For each smooth and compactly supported function $f$,  $Tf$ admits the following representation 
\[
Tf(x)=\int_{\mathbb{R}^{n}}  K(x,y)f(y)dy \qquad x\not\in\supp f
\]
where $K$ is a standard kernel. Recall that a kernel $K:\mathbb{R}^{n}\times\mathbb{R}^{n}\setminus\Delta\longrightarrow\mathbb{R}$,  where $\Delta$ is the diagonal in $\mathbb{R}^{n}\times\mathbb{R}^{n}$, is a locally integrable function such that for some constants $C_{1},C_{2},\gamma>0$ the following conditions hold:
\begin{enumerate}
\item Size condition\[\left|K(x,y)\right|\leq C_{1}\frac{1}{|x-y|^{n}}\qquad \text{if } x\not=y.\]
\item Regularity condition \[|K(x,y)-K(x',y)|+|K(y,x')-K(y,x)|\leq C_{2}\frac{|x-x'|^{\gamma}}{|x-y|^{n+\gamma}}\qquad\]
provided that $|x-x'|\leq\frac{1}{2}|x-y|$.
\end{enumerate}
\end{enumerate}

The symbol-multilinear commutator \,$T_{\vec{b}}$\, with vector symbol $\vec{b}=(b_1,\cdots,b_k)$,  $b_{i}\in Osc_{\exp L^{s_{i}}}, i=1,\cdots,k$, and CZO $T$ with kernel $K$ is defined for smooth functions $f$ as follows 
\[
T_{\vec{b}}f(x)=\int_{\mathbb{R}^{n}} \prod_{i=1}^{k}\left(b_{i}(x)-b_{i}(y)\right)K(x,y)f(y)dy \qquad x\not\in\supp f. 
\]

Let $b=\{b_{1},b_{2},\dots,b_{k}\}$ be a set of symbols with $b_{i}\in Osc_{\exp L^{s_{i}}}, i=1,\cdots,k$. Also, let $b=\sigma\cup\sigma'$ where $\sigma$ and $\sigma'$ are
pairwise disjoint sets be a splitting of $b$.  If we identify $i$ and $b_{i}$ we can introduce the following notation 
\[
\left(b(x)-\lambda\right)_{\sigma}=\prod_{i\in\sigma}\left(b_{i}(x)-\lambda_{i}\right)
\]
where $\lambda=(\lambda_{1},\lambda_{2},\dots,\lambda_{k})$ and also
to write $\sum_{i\in\sigma}\frac{1}{s_{i}}.$

By $C_{j}(b)$ we refer the family of all the subsets $\sigma$ of
$b$ such that $\#\sigma=j$. We shall also omit the set of symbols and write just $C_j^k$. Finally if $\sigma$ is a subset of $b$ we write
\[
T_{\vec{\sigma}}f(x)=\int_{\mathbb{R}^{n}} \prod_{i\in\sigma}\left(b_{i}(x)-b_{i}(y)\right)K(x,y)f(y)dy=\int_{\mathbb{R}^{n}} \left(b(x)-b(y)\right)_{\sigma} K(x,y)f(y)dy \qquad x\not\in\supp f. 
\]

We end this section with some further notation. We write 
\[\|\vec{b}\|=\prod_{b_i \in b}\|b_i\|_{Osc_{\exp{L^{s_i}}}}\] 
and similarly
\[\left\|\vec{\sigma}\right\|=\prod_{b_i\in \sigma}\|b_i\|_{Osc_{\exp{L^{s_i}}}}.\]
We will denote by $\#\sigma$ the cardinal of the set of symbols $\sigma$.

\subsection{Some estimates involving the sharp function}
In this paper we will use two classical operators and some of their variations. The first one is the Hardy-Littlewood maximal operator, 
\[Mf(x)=\sup_{x\in Q}\frac{1}{|Q|}\int_{Q} |f(y)|dy,\]
where each $Q$ is a cube with sides parallel to the axis. Also, $M^d$ will denote its dyadic version, where the supremum is taken over dyadic cubes. We will also use the following variants, 
$M_\varepsilon(f)=M(|f|^\varepsilon)^\frac{1}{\varepsilon}$,  and similarly for $M_\varepsilon^d$ where $\varepsilon\in (0,\infty)$.
The second operator is the Fefferman-Stein sharp maximal function, namely
\[
M^{\sharp}f(x)=\sup_{x\in Q}\frac{1}{|Q|}\int_{Q}\left|f(y)-f_{Q}\right|dy,
\]
and its dyadic counterpart $M^{\sharp,d}$. Similarly as above we define the following useful variation
$$M_\delta^\sharp(f)=M^\sharp(|f|^\delta)^\frac{1}{\delta}$$ 
with $\delta\in(0,\infty)$. 

The first result that we state in this section is borrowed from \cite{MR3060421}.
\begin{lem}\label{OCPR}
Let $0<p<\infty$, $0<\delta<1$ and let $w\in A_{\infty}$. Then
\[
\|f\|_{L^{p}(w)}\leq cp[w]_{A_{\infty}}\left\Vert M_{\delta}^{\sharp,d}f\right\Vert _{L^{p}(w)}
\] 
for any function $f$ such that $\left|\left\{ x\,:\,|f(x)|>t\right\} \right|<\infty$
for all $t>0$.
\end{lem}
Using the preceding lemma and following the proof of Lemma 3.1 in \cite{MR3008263} we can derive the following improvement. 
\begin{lem}\label{Lema3Carmen}
Let $0<p<\infty$, $0<\varepsilon\leq1$ and $w\in A_{\infty}$. Suppose
that 
\[
\left|\left\{ x\,:\,|f(x)|>t\right\} \right|<\infty
\]
for all $t>0$. Then there is a constant $c=c_{n,\varepsilon}$ such
that 
\[
\left\Vert M_{\varepsilon}^{d}f\right\Vert _{L^{p}(w)}\leq cp[w]_{A_{\infty}}\left\Vert M_{\varepsilon}^{\sharp,d}f\right\Vert _{L^{p}(w)} 
\] 
\end{lem}
\begin{proof}
Applying previous lemma with $\delta=\varepsilon_{0}$ and $0<\varepsilon_{0}<\varepsilon<1$
\[
\left\Vert M_{\varepsilon}^{d}f\right\Vert _{L^{p}(w)}\leq cp[w]_{A_{\infty}}\left\Vert M_{\varepsilon_{0}}^{\sharp,d}\left(M_{\varepsilon}^{d}f\right)\right\Vert _{L^{p}(w)}.
\]
Now it suffices to prove that 
\[
M_{\varepsilon_{0}}^{\sharp,d}\left(M_{\varepsilon}^{d}f\right)(x)\leq cM_{\varepsilon}^{\sharp,d}f(x).
\]
But this was done in Lemma 3.1 of \cite{MR3008263}.\end{proof}

The reason why it is important to deal with $M_{\varepsilon}^{\sharp}$ for small $\varepsilon$ will be clear after the following pointwise estimate proved in \cite{MR1895740}.

\begin{lem}
\label{MsharpPuntual}Let
$T_{\vec{b}}$  be the symbol-multilinear
commutator defined above and let \,$0<\delta<\varepsilon<1$.
Then there exists a constant $c>0$, depending only on $\delta$ and
$\varepsilon$ such that 
\[
M_{\delta}^{\sharp}\left(T_{\vec{b}}f\right)(x)\leq c_{\delta,\varepsilon}\left(\|\vec{b}\|M_{L\left(\log L\right)^{\frac{1}{s}}} (f) + \sum_{j=1}^{k}\sum_{\sigma \in C_{j}^{k}}\|\vec{\sigma}\|M_{\varepsilon}\left(T_{\vec{b_{\sigma'}}}f\right)(x)\right)
\]
for any bounded function $f$ with compact support.

\end{lem}

 The first result, corresponding to the case  $k=0$, namely 
\begin{equation}\label{AP}
M_{\delta}^{\sharp}(Tf)(x)\leq c_{\delta}\,Mf(x) \qquad 0<\delta<1
\end{equation}
can be found in \cite{MR1273194} and the case $k=1$ was established in \cite{MR1317714}.

\section{Proof of Theorem \ref{TipoFuerte}}\label{PTF}

In this section we prove Theorem \ref{TipoFuerte}. The first two subsections will be devoted  to the core of the proof for all the cases, namely $k=1$ and $k>1$. Both cases rely upon a two weight inequality that will be  established in the third subsection. A careful control of some Young's functions inverses will be required to obtain that two weight inequality.

\subsection{Case $k=1$}
\begin{proof}
In this proof we follow techniques in \cite{MR2480568} and \cite{MR3008263}. Let us call $v=M_{L(\log L)^{\left(1+\frac{1}{s}\right)p-1+\delta}}w$.
If $\kappa=c_{T}\left(p'\right)^{2}p^{1+\frac{1}{s}}\left(\frac{p-1}{\delta}\right)^{\frac{1}{p'}}$,
by duality, it suffices to show that 
\[
\left\Vert \frac{[b,T]^{t}f}{v}\right\Vert _{L^{p'}(v)}\leq\kappa\left\Vert \frac{f}{w}\right\Vert _{L^{p'}(w)}
\]
Where $[b,T]^{t}$ is the adjoint of $[b,T]$. Calculating the norm
by duality allows us to find a  non-negative function  $h\in L^{p}(v)$ with $\|h\|_{L^{p}(v)}=1$
such that 
\[
\left\Vert \frac{[b,T]^{t}f}{v}\right\Vert _{L^{p'}(v)}=\int_{\mathbb{R}^{n}}\frac{ |[b,T]^{t}f| }{v}\,hvdx=\int_{\mathbb{R}^{n}} |[b,T]^{t}| f\,hdx=I.
\]
Let us consider the operator 
\[
S(h)=\frac{M\left(hv^{\frac{1}{p}}\right)}{v^{\frac{1}{p}}}.
\]
We build the Rubio de Francia algorithm $R$ using the operator $S$.
\[
R(h)=\sum_{k=0}^{\infty}\frac{1}{2^{k}}\frac{S^{k}(h)}{\|S\|_{L^{p}(v)}^{k}}.
\]
$R$ satisfies the following properties:
\begin{enumerate}
\item $0\leq h\leq R(h)$
\item $\|R(h)\|_{L^{p}(v)}\leq2\|h\|_{L^{p}(v)}$
\item $R(h)v{}^{\frac{1}{p}}\in A_{1}$ and furthermore $\left[R(h)v{}^{\frac{1}{p}}\right]_{A_{1}}\leq cp'$.
\end{enumerate}

Using Lemma 4.2 in \cite{MR3327006}
\[
\left[v{}^{\frac{1}{2p}}\right]_{A_{1}}^{2}\leq c_{n}
\]
 Taking that into account 
\[
[Rh]_{A_{3}}=\left[R(h)v{}^{\frac{1}{p}}\left(v^{-\frac{1}{p(1-q)}}\right)^{1-q}\right]_{A_{q}}\leq\left[R(h)v{}^{\frac{1}{p}}\right]_{A_{1}}\left[v{}^{\frac{1}{2p}}\right]_{A_{1}}^{2}\leq c_{n}p'
\]

Applying Lemma \ref{OCPR}, with $p=1$, $w=Rh$ and $\gamma\in(0,1)$, together with 
$[Rh]_{A_{\infty}}\leq[Rh]_{A_{3}}\leq c_{n}p'$, we have 
\[
\begin{split}I & \leq\int_{\mathbb{R}^{n}}\left|[b,T]^{t}f\right|Rhdx\\
 & \leq c_{n}[Rh]_{A_{\infty}}\int_{\mathbb{R}^{n}}M_{\gamma}^{\sharp}([b,T]^{t}f)Rh(x)dx\leq c_{n}p'\int_{\mathbb{R}^{n}}M_{\gamma}^{\sharp}([b,T]^{t}f)Rh(x)dx
\end{split}
\]
Now we observe that $[b,T]^{t}=-\left[b,T^{t}\right]$. Consequently,
$[b,T]^{t}$ is a commutator and if we choose $\varepsilon\in(\delta,1)$ in Lemma \ref{MsharpPuntual} we can continue with 
\[
\begin{split} 
 & \leq c_{n}p'\|b\|_{Osc_{expL^{s}}}\left[\int_{\mathbb{R}^{n}}M_{L(\log L)^{\frac{1}{s}}}f(x)Rh(x)dx+\int_{\mathbb{R}^{n}}M_{\varepsilon}\left(T^{t}f\right)(x)Rh(x)dx\right]\\
 & =c_{n}p'\|b\|_{Osc_{expL^{s}}}(I_{1}+I_{2})
\end{split}
\]
To estimate $I_{1}$ we use  H\"older's inequality and the second property of the operator $R$
\[
\begin{split}I_{1} & =\int_{\mathbb{R}^{n}} M_{L(\log L)^{\frac{1}{s}} } f(x)Rh(x)dx\leq\left(\int_{\mathbb{R}^{n}} M_{L(\log L)^{\frac{1}{s}} } f(x)^{p'}v(x)^{1-p'}dx\right)^{\frac{1}{p'}}\left(\int_{\mathbb{R}^{n}}Rh(x)^{p}v(x)dx\right)^{ \frac1p } \\
 & \leq2\left\Vert \frac{M_{L(\log L)^{\frac{1}{s}}}f}{v}\right\Vert _{L^{p'}(v)}
\end{split}
\]
To bound $I_{2}$ we apply lemma \ref{Lema3Carmen}
with $w=Rh$ and $p=1$
\[
I_{2}\leq c_{n}[Rh]_{A_{ \infty} }\int_{\mathbb{R}^{n}}M_{\varepsilon}^{\sharp}(T^{t}f)(x)Rh(x)dx\leq c_{n}p'\int_{\mathbb{R}^{n}}M_{\varepsilon}^{\sharp}(T^{t}f)(x)Rh(x)dx.
\]
Using now \eqref{AP} for the adjoint of $T$, namely,  $M_{\varepsilon}^{\sharp}(T^{t}f)\leq c_{\varepsilon}\,Mf$, we have
\[
I_{2} \leq c_{n,\varepsilon}p'\int_{\mathbb{R}^{n}}MfRh.
\]
Proceeding now as we did for $I_{1}$ we derive to
\[
I_{2}\leq c_{n}p'\left\Vert \frac{Mf}{v}\right\Vert _{L^{p'}(v)}.
\]
Consequently 
\[
\left\Vert \frac{[b,T]^{t}f}{v}\right\Vert _{L^{p'}(v)}\leq c_{n}\left(p'\right)^{2}\|b\|_{Osc_{expL^{s}}}\left\Vert \frac{M_{L(\log L)^{\frac{1}{s}}}f}{v}\right\Vert _{L^{p'}(v)},
\]
and recalling that  $v=M_{L(\log L)^{\left(1+\frac{1}{s}\right)p-1+\delta}}w$, everything is reduced to establish the following inequality,
\begin{equation}
\left\Vert \frac{M_{L(\log L)^{\frac{1}{s}}}f}{v}\right\Vert _{L^{p'}(v)}\leq c_{n}p^{1+\frac{1}{s}}\left(\frac{p-1}{\delta}\right)^{\frac{1}{p'}}\left\Vert \frac{f}{w}\right\Vert _{L^{p'}(w)}
\end{equation}
which will be proved in Lemma \ref{lem:2Weight} below. This concludes the proof of the Theorem in the case $k=1$. 

\end{proof}

\subsection{Case $k>1$}
\begin{proof}
Due to the homogeneity of the operator we may assume that 
\[
\|b_1\|_{Osc_{expL^{s_{1}}}}=\|b_2\|_{Osc_{expL^{s_{2}}}}=\dots=\|b_k\|_{Osc_{expL^{s_{k}}}}=1
\]
Let us denote $v=M_{L(\log L)^{\left(1+\frac{1}{s}\right)p-1+\delta}}w$.
If $\kappa=c_{T}\left(p'\right)^{k+1}p^{1+\frac{1}{s}}\left(\frac{p-1}{\delta}\right)^{\frac{1}{p'}}$,
by duality, it suffices to show that 
\[
\left\Vert \frac{T_{\vec{b}}^{t}f}{v}\right\Vert _{L^{p'}(v)}\leq\kappa\left\Vert \frac{f}{w}\right\Vert _{L^{p'}(w)},
\]
where $T_{\vec{b}}^{t}$ is the adjoint of $T_{\vec{b}}$. Using duality we can find a non-negative function $h\in L^{p}(v)$ with $\|h\|_{L^{p}(v)}=1$ such that 
\[
\left\Vert \frac{T_{\vec{b}}^{t}f}{v}\right\Vert _{L^{p'}(v)}= \int_{\mathbb{R}^{n}}\frac{|T_{\vec{b}}^{t}f|}{v}\,hvdx=
\int_{\mathbb{R}^{n}}|T_{\vec{b}}^{t}f|\,hdx=I.
\]
As in the case $k=1$ we use again Lemma \ref{OCPR} with $p=1$, $w=Rh$ and $\gamma\in (0,1)$. Hence, since $[Rh]_{A_{\infty}}\leq[Rh]_{A_{3}}\leq c_{n}p'$, we have
\[
\begin{split}I & \leq\int_{\mathbb{R}^{n}}\left|T_{\vec{b}}^{t}f\right| \,Rhdx\\
 & \leq c_{n}[Rh]_{A_{\infty}}\int_{\mathbb{R}^{n}}M_{\gamma}^{\sharp}(T_{\vec{b}}^{t}f)\,Rhdx\leq c_{n}p'\int_{\mathbb{R}^{n}}M_{\gamma}^{\sharp}(T_{\vec{b}}^{t}f)\,Rhdx
\end{split}
\]
 $T_{\vec{b}}^{t}f$ is a commutator so if we take $\varepsilon\in(\gamma,1)$ Lemma \ref{MsharpPuntual} yields
\[
\begin{split} & c_{n}p'\int_{\mathbb{R}^{n}} M_{\gamma}^{\sharp}\left(T_{\vec{b}}^{t}f\right) \,Rh(x)dx\\
 & \leq c_{n}p'\left[\int_{\mathbb{R}^{n}}M_{L(\log L)^{\frac{1}{s}}}f(x)\,Rhdx+\sum_{j=1}^{k}\sum_{\sigma\in C_{j}^{k}}\int_{\mathbb{R}^{n}}M_{\varepsilon}\left(T_{\vec{\sigma'}}^tf\right)\,Rhdx\right]\\
 & =c_{n}p'(I_{1}+I_{2})
\end{split}
\]
Now we have to estimate $I_{1}$ and $I_{2}$. For $I_{1}$ we proceed as in the case $k=1$ obtaining 
\[
I_{1}\leq2\left\Vert \frac{M_{L(\log L)^{\frac{1}{s}}}f}{v}\right\Vert _{L^{p'}(v)}
\]
To estimate $I_{2}$ we need to control each term of the sum.
%
To accomplish this we claim that for every $\varepsilon\in(0,1)$:
\begin{equation}\label{claim}
\int_{\mathbb{R}^{n}}M_{\varepsilon}\left(T_{\vec{\sigma'}}^{t}f\right)(x)Rh(x)dx\leq c \Gamma(\#\sigma')(p')^{\#\sigma'+1}\int_{\mathbb{R}^{n}}M_{L(\log L)^{\sum_{i\in\sigma'}\frac{1}{s_{i}}}}f(x)Rh(x)dx
\end{equation}
where 
\[
\Gamma(j)=\begin{cases}
1 & \quad j=0\\
2 & \quad j=1\\
2+\sum_{i=1}^{j-1}\binom{j}{i}\Gamma(i) & \quad j>1
\end{cases}
\]
which will be proved by induction on the number of symbols of $T_{\vec{\sigma'}}^t$, i.e., $\#\sigma'$. For simplicity with the notation, we prove the claim for $T_{\vec{\sigma'}}$ instead of $T_{\vec{\sigma'}}^t$. We can do this since both them are commutators with the same number of symbols.
Let us call $m=\#\sigma'$. If the number of symbols is zero $T_{\vec{\sigma'}}=T$ and then combining Lemma \ref{Lema3Carmen} and \eqref{AP} we obtain
\[
\int_{\mathbb{R}^{n}}M_{\varepsilon}\left(Tf\right)(x)Rh(x)dx\leq cp'\int_{\mathbb{R}^{n}}Mf(x)Rh(x)dx
\]
%
since we assume $\sum_{i\in\emptyset}\frac{1}{s_{i}}=0$. 
If $m=1$, then $T_{\vec{\sigma'}}=[b_{1},T]$.
Applying Lemma \ref{Lema3Carmen} with $w=Rh$, $p=1$
we have that 
\[
\begin{split}\int_{\mathbb{R}^{n}}M_{\varepsilon}\left([b_{1},T]f\right)(x)\,Rhdx & \leq c_{n}[Rh]_{A_{3}}\int_{\mathbb{R}^{n}}M_{\varepsilon}^{\sharp}([b_{1},T]f)\,Rhdx\\
 & \leq c_{n}p'\int_{\mathbb{R}^{n}}M_{\varepsilon}^{\sharp}([b_{1},T]f)\,Rhdx
\end{split}
\]
Now, if we take $0<\varepsilon<\tilde{\varepsilon}<1$,  Lemma \ref{MsharpPuntual} produces the following bound of the last expression
\[
\begin{split} 
 & \leq cp'\int_{\mathbb{R}^{n}}M_{L(\log L)^{\frac{1}{s_{1}}}}f(x)Rh(x)dx+cp'\int_{\mathbb{R}^{n}}M_{\tilde{\varepsilon}}(Tf(x))Rh(x)dx\\
 & \leq cp'\int_{\mathbb{R}^{n}}M_{L(\log L)^{\frac{1}{s_{1}}}}f(x)Rh(x)dx+c(p')^{2}\int_{\mathbb{R}^{n}}Mf(x)Rh(x)dx\\
 & \leq 2c(p')^{2}\int_{\mathbb{R}^{n}}M_{L(\log L)^{\frac{1}{s_{1}}}}f(x)Rh(x)dx.
\end{split}
\]
This proves claim \eqref{claim} in the case $m=1$. Let us suppose now that the result holds for $0\leq l<m$ symbols, 
namely, if $0\leq\#\tau<m$, then for every $\varepsilon\in(0,1)$
\[
\int_{\mathbb{R}^{n}}M_{\varepsilon}\left(T_{\vec{\tau}}f\right)(x)\,Rhdx\leq c\Gamma(\#\tau)(p')^{\#\tau+1}\int_{\mathbb{R}^{n}}M_{L(\log L)^{\sum_{i\in\tau}\frac{1}{s_{i}}}}f(x)\,Rhdx
\]
Combining Lemma \ref{Lema3Carmen}, with $w=Rh$ and $p=1$, 
and Lemma \ref{MsharpPuntual}  we have for $\varepsilon<\tilde{\varepsilon}<1$,
\[
\begin{split} & \int_{\mathbb{R}^{n}}M_{\varepsilon}\left(T_{\vec{\sigma'}}f\right)\,Rhdx\\
 & \leq cp'\int_{\mathbb{R}^{n}} M_{L(\log L)^{\sum_{_{i\in\sigma'}}\frac{1}{s_{i}  }}}f\,Rhdx+\sum_{j=1}^{m}\sum_{\tau\in C_{j}^{m}} 
 cp'\int_{\mathbb{R}^{n}}M_{\tilde{\varepsilon}}\left(T_{\vec{b_{\tau'}}}f\right)\,Rhdx\\
 & \leq cp'\int_{\mathbb{R}^{n}}M_{L(\log L)^{\sum_{_{i\in\sigma'}}\frac{1}{s_{i}}}}f\,Rhdx
 +\sum_{j=1}^{k}\sum_{\tau\in C_{j}^{k}}cp'\int_{\mathbb{R}^{n}}M_{\tilde{\varepsilon}}\left(T_{\vec{b_{\tau'}}}f\right)\,Rhdx.
\end{split}
\]
Using now the induction hypothesis we continue with
\[
\begin{split} 
 & \leq cp'\int_{\mathbb{R}^{n}}M_{L(\log L)^{\sum_{_{i\in\sigma'}}\frac{1}{s_{i}}}}f(x)\,Rhdx+\sum_{j=1}^{m}\sum_{\tau\in C_{j}^{m}}\Gamma(\#\tau')c(p')^{\#\tau'+1}\int_{\mathbb{R}^{n}}M_{L(\log L)^{\sum_{i\in\tau'}\frac{1}{s_{i}}}}f(x)\,Rhdx\\
 & \leq cp'\int_{\mathbb{R}^{n}}M_{L(\log L)^{\sum_{_{i\in\sigma' }}\frac{1}{s_{i}}}}f(x)\,Rhdx+\left(\sum_{j=1}^{m}\sum_{\tau\in C_{j}^{m}}\Gamma(\#\tau')\right)(p')^{m+1}\int_{\mathbb{R}^{n}}M_{L(\log L)^{\sum_{i\in\sigma'}\frac{1}{s_{i}}}}f(x)\,Rhdx\\
 & \leq \left(1 + \sum_{j=1}^{m}\sum_{\tau\in C_{j}^{m}}\Gamma(\#\tau')\right) c(p')^{k+1}\int_{\mathbb{R}^{n}}M_{L(\log L)^{\sum_{i\in\sigma'}\frac{1}{s_{i}}}}f(x)\,Rhdx.
\end{split}
\]
It's easy to check that \[1 + \sum_{j=1}^{m}\sum_{\tau\in C_{j}^{m}}\Gamma(\#\tau')=\Gamma(m)\]
Then the main claim \eqref{claim} is proved in full generality.  This yields, combining estimates 
\[
I_{2} \leq c_{n,k,\delta,\varepsilon}\left(p'\right)^{k+1}\int_{\mathbb{R}^{n}}M_{L(\log L)^{\frac{1}{s}}}f(x)\,Rhdx.
\]
Proceeding as we did for $I_{1}$ we obtain the following estimate
\[
I_{2}\leq c_{n,\delta,\varepsilon}\left(p'\right)^{k+1}\left\Vert \frac{M_{L(\log L)^{\frac{1}{s}}}f}{v}\right\Vert _{L^{p'}(v)}.
\]
Consequently 
\[
\left\Vert \frac{T_{\vec{b}}^{t}f}{v}\right\Vert _{L^{p'}(v)}\leq c_{n}\left(p'\right)^{k+1} \left\Vert \frac{M_{L(\log L)^{\frac{1}{s}}}f}{v}\right\Vert _{L^{p'}(v)}.
\]
This concludes the proof since by Lemma \ref{lem:2Weight} 
\[
\left\Vert \frac{M_{L(\log L)^{\frac{1}{s}}}f}{v}\right\Vert _{L^{p'}(v)}\leq cp^{1+\frac{1}{s}}\left(\frac{p-1}{\delta}\right)^{\frac{1}{p'}}\left\Vert \frac{f}{w}\right\Vert _{L^{p'}(w)}
\]
 that $v=M_{L(\log L)^{\left(1+\frac{1}{s}\right)p-1+\delta}}w$.
\end{proof}
\subsection{A key two weight inequality}

As we already mentioned, we prove the following lemma that was used several times during the proof of Theorem \ref{TipoFuerte}.

\begin{lem}\label{lem:2Weight}
Let $w\geq 0$ be a weight. Let $s \geq  1$ and $0<\delta<1$. Then for every $p\in(1,\infty)$ we have that 
\begin{equation}\label{claimSpecial2weight}
\left\Vert \frac{M_{L(\log L)^{\frac{1}{s}}}f}{v}\right\Vert _{L^{p'}(v)}\leq c p^{1+\frac{1}{s}}\left(\frac{p-1}{\delta}\right)^{\frac{1}{p'}}\left\Vert \frac{f}{w}\right\Vert _{L^{p'}(w)}.
\end{equation}
where $v=M_{L(\log L)^{\left(1+\frac{1}{s}\right)p-1+\delta}}w$.
\end{lem}

The proof of this lemma will follow ideas from \cite{MR1481632}. In particular,  we are going to obtain a precise version of the two weight inequality that appears in the proof of  Theorem 2 of that work. To do that we need precise estimates of the following inverse functions.

\begin{lem}
\label{Lema2}Let $\rho>0$, $A_{\rho}(t)=t\left(1+\log^{+}\left(t\right)\right)^{\rho}$
and $X_{\rho}(t)=\frac{t}{\left(1+\log^{+}\left(t\right)\right)^{\rho}}$.
Then 
\[
\left(\frac{1}{1+\rho}\right)^{\rho}t\leq X_{\rho}(A_{\rho}(t))\leq t.
\]
\end{lem}
\begin{proof}
Observe that 
\[
X_{\rho}(A_{\rho}(t))=\frac{t\left(1+\log^{+}\left(t\right)\right)^{\rho}}{\left(1+\log^{+}\left(t\left(1+\log^{+}\left(t\right)\right)^{\rho}\right)\right)^{\rho}}
\]
The upper bound is straightforward since 
\[
\left(1+\log^{+}\left(t\right)\right)^{\rho}\leq\left(1+\log^{+}\left(t\left(1+\log^{+}\left(t\right)\right)^{\rho}\right)\right)^{\rho}.
\]
Now we prove the lower bound. It suffices to prove that 
\[
\frac{1+\log^{+}\left(t\right)}{1+\log^{+}\left(t\left(1+\log^{+}\left(t\right)\right)^{\rho}\right)}\geq\frac{1}{1+\rho}.
\]
If $0<t\leq1$ there's nothing to prove since $\log^{+}\left(t\right)=\log^{+}\left(t\left(1+\log^{+}\left(t\right)\right)^{\rho}\right)=0$.
Suppose now that $t>1$. Then we have that 
\[
\begin{split} & \frac{1+\log^{+}\left(t\right)}{1+\log^{+}\left(t\left(1+\log^{+}\left(t\right)\right)^{\rho}\right)}=\frac{1+\log\left(t\right)}{1+\log\left(t\left(1+\log\left(t\right)\right)^{\rho}\right)}\\
 & =\frac{1+\log\left(t\right)}{1+\log\left(t\right)+\rho\log\left(1+\log\left(t\right)\right)}\geq\frac{1+\log\left(t\right)}{1+\log\left(t\right)+\rho\left(1+\log\left(t\right)\right)}\\
 & =\frac{1}{1+\rho}.
\end{split}
\]
 
\end{proof}

\begin{lem}
\label{Lema3}Let $\rho>1$, $A_{\rho}(t)=t\left(1+\log^{+}\left(t\right)\right)^{\rho}$
and $\tilde{X_{\rho}}(t)=\frac{t}{\left(1+\log^{+}\left(\frac{t}{t_{\rho}}\right)\right)^{\rho}}$
with $t_{\rho}=\rho^{\rho}$. Then 
\[
\left(1-\frac{1}{e}\right)^{\rho}t\leq A_{\rho}(\tilde{X_{\rho}}(t))\leq t\left(1+\rho\log\left(\rho\right)\right)^{\rho}.
\]
\end{lem}
\begin{proof}
Observe first that
\[
A_{\rho}(\tilde{X_{\rho}}(t))=t\left(\frac{1+\log^{+}\left(\frac{t}{\left(1+\log^{+}\left(\frac{t}{t_{\rho}}\right)\right)^{\rho}}\right)}{1+\log^{+}\left(\frac{t}{t_{\rho}}\right)}\right)^{\rho}=t\Phi(t)^{\rho}
\]
We begin studying the lower bound. \\
If $t\in(0,1)$ then 
\[
A_{\rho}(\tilde{X_{\rho}}(t))=t\Phi(t)^{\rho}=t
\]
and there's nothing to prove. \\
If $t\in[1,t_{\rho}]$ then 
\[
A_{\rho}(\tilde{X_{\rho}}(t))=t\Phi(t)^{\rho}=t\left(1+\log^{+}\left(t\right)\right)^{\rho}\geq t
\]
Now if $t>t_{\rho}$, it's easy to check that $\frac{t}{\left(1+\log\left(\frac{t}{t_{\rho}}\right)\right)^{\rho}}\geq1$.
Then 
\[
A_{\rho}(\tilde{X_{\rho}}(t))=t\left(\frac{1+\log\left(\frac{t}{\left(1+\log\left(\frac{t}{t_{\rho}}\right)\right)^{\rho}}\right)}{1+\log\left(\frac{t}{t_{\rho}}\right)}\right)^{\rho}
\]
Now we observe that
\[
\frac{1+\log\left(\frac{t}{\left(1+\log\left(\frac{t}{t_{\rho}}\right)\right)^{\rho}}\right)}{1+\log\left(\frac{t}{t_{\rho}}\right)}=\frac{1+\log\left(t\right)-\rho\log\left(1+\log\left(\frac{t}{t_{\rho}}\right)\right)}{1+\log\left(\frac{t}{t_{\rho}}\right)}.
\]
Let us choose $t=e^{\lambda}$ and $t_{\rho}=e^{\lambda_{\rho}}$.
Then
\[
\frac{1+\lambda-\rho\log\left(1+\log\left(\frac{e^{\lambda}}{e^{\lambda_{\rho}}}\right)\right)}{1+\log\left(\frac{e^{\lambda}}{e^{\lambda_{\rho}}}\right)}=1+\frac{\lambda_{\rho}-\rho\log\left(1+\lambda-\lambda_{\rho}\right)}{1+\lambda-\lambda_{\rho}}=1+g_{\rho}(\lambda)
\]
Now we minimize $g_{\rho}(\lambda)$.  It's easy to check that
$g_{\rho}$ reaches its minimum when $\lambda=e^{1+\frac{\lambda_{\rho}}{\rho}}+\lambda_{\rho}-1$.
We observe that
\[
g_{\rho}\left(e^{1+\frac{\lambda_{\rho}}{\rho}}+\lambda_{\rho}-1\right)=\frac{-\rho}{e^{1+\frac{\lambda_{\rho}}{\rho}}}
\]
and since $t_{\rho}=\rho^{\rho}$ 
\[
\frac{-\rho}{e^{1+\frac{\lambda_{\rho}}{\rho}}}=-\frac{1}{e}
\]
and we obtain the desired lower bound. To finish the proof we focus
on the bound. If $t\in(0,1),$ then $A_{\rho}(\tilde{X_{\rho}}(t))=t$
and there's nothing to prove. If $t\in\left[1,t_{\rho}\right]$ then
we have that 
\[
A_{\rho}(\tilde{X_{\rho}}(t))=t(1+\log t)^{\rho}\leq t(1+\log t_{\rho})^{\rho}=t(1+\rho\log\rho)^{\rho}.
\]
Finally if $t\in\left(t_{\rho},\infty\right)$ then it's easy to check
that
\[
A_{\rho}(\tilde{X_{\rho}}(t))\leq t\left(1+\log\left(t_{\rho}\right)\right)^{\rho}.
\]
\end{proof}

Finally, with the precise control of the inverses at our disposal we are ready to give the proof of lemma \ref{lem:2Weight}. 

\begin{proof}[Proof of Lemma \ref{lem:2Weight}]
Proving
\eqref{claimSpecial2weight} is equivalent to prove that 
\[ \label{Eq2weight}
\int_{\mathbb{R}^{n}}M_{L\log L^{\frac{1}{s}}}\left(fw^{\frac{1}{p}}\right)^{p'}\left(M_{L(\log L)^{\left(1+\frac{1}{s}\right)p-1+\delta}}w\right)^{1-p'}\leq c_{n}^{p'}\left(p^{1+\frac{1}{s}}\right)^{p'}\left(\frac{p-1}{\delta}\right)\int_{\mathbb{R}^{n}}|f|^{p'}
\]
Using now the notation of Lemma \ref{Lema2}, we can write $A_{\frac{1}{s}}(t)=t(1+\log^{+}t)^{\frac{1}{s}}$
and $X_{\frac{1}{s}}(t)=\frac{t}{\left(1+\log^{+}t\right)^{\frac{1}{s}}}$
and we have that 
\[
A_{\frac{1}{s}}^{-1}(t)\geq X_{\frac{1}{s}}(t)
\]
We observe now that 
\[
\begin{split}X_{\frac{1}{s}}(t) & =\frac{t}{\left(1+\log^{+}t\right)^{\frac{1}{s}}}=\frac{t^{\frac{1}{p}}}{\left(1+\log^{+}t\right)^{\frac{1}{s}+\frac{p-1+\delta}{p}}}\cdot t^{\frac{1}{p'}}\left(1+\log^{+}t\right)^{\frac{p-1+\delta}{p}}\\
 & =\left(\frac{t}{\left(1+\log^{+}t\right)^{\left(1+\frac{1}{s}\right)p-1+\delta}}\right)^{\frac{1}{p}}\left(t\left(1+\log^{+}t\right)^{1+\delta(p'-1)}\right)^{\frac{1}{p'}}=F_{1}(t)^{\frac{1}{p}}\cdot F_{2}(t)^{\frac{1}{p'}}
\end{split}
\]
Using again the notation of Lemma \ref{Lema2}, 
\[
F_{1}(t)=X_{\left(1+\frac{1}{s}\right)p-1+\delta}(t)=\frac{t}{\left(1+\log^{+}t\right)^{\left(1+\frac{1}{s}\right)p-1+\delta}}.
\]
From that lemma it readily follows that  
\[
F_{1}(t)^{\frac{1}{p}}\geq\left(\frac{1}{\left(1+\frac{1}{s}\right)p+\delta}\right)^{\frac{\left(1+\frac{1}{s}\right)p-1+\delta}{p}}A_{\left(1+\frac{1}{s}\right)p-1+\delta}^{-1}(t)^{\frac{1}{p}}.\label{F1}
\]
Analogously, following the notation of Lemma \ref{Lema3}
\[
F_{2}(t)=A_{1+\delta(p'-1)}(t)=t\left(1+\log^{+}t\right)^{1+\delta(p'-1)}
\]
From that lemma it follows that  
\[
F_{2}(t)^{\frac{1}{p'}}\geq\left(\frac{e-1}{e}\right)^{\frac{1+\delta(p'-1)}{p'}}\tilde{X}_{1+\delta(p'-1)}^{-1}(t)^{\frac{1}{p'}}.\label{F2}
\]
Taking into account \eqref{F1} and \eqref{F2} we obtain the following estimate 
\[
A_{\frac{1}{s}}^{-1}(t)\left(e'\right)^{\frac{1+\delta(p'-1)}{p'}}\left(\left(1+\frac{1}{s}\right)p+\delta\right)^{\frac{\left(1+\frac{1}{s}\right)p-1+\delta}{p}}\geq A_{\left(1+\frac{1}{s}\right)p-1+\delta}^{-1}(t)^{\frac{1}{p}}\tilde{X}_{1+\delta(p'-1)}^{-1}(t)^{\frac{1}{p'}}\qquad t>0.
\]
Using now generalized Hölder inequality (Lemma \ref{Lemma:HolderGeneralizado}) and taking into account that, since $\delta\in(0,1)$,
\[
\left(e'\right)^{\frac{1+\delta(p'-1)}{p'}}\left(2p+\delta\right)^{\frac{\left(1+\frac{1}{s}\right)p-1+\delta}{p}}\leq cp^{1+\frac{1}{s}}
\]
and also that $\left\Vert w\right\Vert _{\Psi(L)}=\left\Vert w^{p}\right\Vert _{\Psi\left(L^{\frac{1}{p}}\right)}^{\frac{1}{p}}$
if $\Psi$ is a Young function, we have that
\[
\left\Vert fw^{\frac{1}{p}}\right\Vert _{L(\log L)^\frac{1}{s},Q}\leq cp^{1+\frac{1}{s}}\left\Vert f\right\Vert _{\tilde{X}_{1+\delta(p'-1)}(L^{p'}),Q}\left\Vert w\right\Vert _{A_{\left(1+\frac{1}{s}\right)p-1+\delta}(L),Q}^{\frac{1}{p}}
\]
and consequently
\[
M_{L{(\log L)^{\frac{1}{s}}}}\left(fw^{\frac{1}{p}}\right)\leq cp^{1+\frac{1}{s}} M_{\tilde{X}_{1+\delta(p'-1)}(L{}^{p'})}(f)M_{L(\log L)^{\left(1+\frac{1}{s}\right)p-1+\delta}}\left(w\right)^{\frac{1}{p}}.
\]
Using this estimate we have that 
\[
\begin{split} & \int_{\mathbb{R}^{n}}M_{L(\log L)^\frac{1}{s}}\left(fw^{\frac{1}{p}}\right)^{p'}\left(M_{L(\log L)^{\left(1+\frac{1}{s}\right)p-1+\delta}}w\right)^{1-p'}dx\\
 & \leq\int_{\mathbb{R}^{n}}\left(cp^{\left(1+\frac{1}{s}\right)}M_{\tilde{X}_{1+\delta(p'-1)}(L{}^{p'})}(f)M_{L(\log L)^{\left(1+\frac{1}{s}\right)p-1+\delta}}\left(w\right)^{\frac{1}{p}}\right)^{p'}\left(M_{L(\log L)^{\left(1+\frac{1}{s}\right)p-1+\delta}}w\right)^{1-p'}dx\\
 & =\left(cp^{1+\frac{1}{s}}\right)^{p'}\int_{\mathbb{R}^{n}}M_{\tilde{X}_{1+\delta(p'-1)}(L{}^{p'})}(f)^{p'}dx
\end{split}
\]
Lemma 2.1 of \cite{MR3327006} yields 
\[
\left(\int_{\mathbb{R}^{n}}M_{\tilde{X}_{1+\delta(p'-1)}(L{}^{p'})}f(x)^{p'}dx\right)^{\frac{1}{p'}}\leq c\left(\frac{p-1}{\delta}\right)^{\frac{1}{p'}}\left(\int_{\mathbb{R}^{n}}|f|^{p'}(x)dx\right)^{\frac{1}{p'}},
\]
since
\[
\left(\int_{1}^{\infty}\frac{\tilde{X}_{1+\delta(p'-1)}(t{}^{p'})}{t^{p'}}\frac{dt}{t}\right)^{\frac{1}{p'}}=\left(\frac{\left(1+(p'-1)\delta\right)}{p'}\log\left(1+(p'-1)\delta\right)+\frac{1}{(p'-1)\delta}\right)^{\frac{1}{p'}}
\]
and $0<\delta<1$ allows us to write 
\[
\left(\frac{\left(1+(p'-1)\delta\right)}{p'}\log\left(1+(p'-1)\delta\right)+\frac{1}{(p'-1)\delta}\right)^{\frac{1}{p'}}\leq c\left(\frac{p-1}{\delta}\right)^{\frac{1}{p'}}.
\]
Consequently we have that
\[
\left\Vert M_{L(\log L)^{\frac{1}{s}}}\left(fw^{\frac{1}{p}}\right)\right\Vert _{L^{p'}(v^{1-p'})}\leq cp^{1+\frac{1}{s}}\left(\frac{p-1}{\delta}\right)^{\frac{1}{p'}}\left\Vert f\right\Vert _{L^{p'}(\mathbb{R}^{n}).}
\]
This concludes the proof of \eqref{Eq2weight}.
\end{proof}

\section{Proof of Theorem \ref{Extremo}}\label{PTD}

\subsection{Case $k=1$}
\begin{proof}
By homogeneity we shall suppose that $\|b\|_{Osc_{expL^{s}}}=1$.
We consider the Calderón-Zygmund decomposition of $f$ at height $\lambda$.
That decomposition allows us to obtain a family of dyadic cubes $\{Q_{j}\}$
which are pairwise disjoint such that 
\[
\lambda\leq\frac{1}{|Q_{j}|}\int_{Q_{j}}|f|\leq2^{n}\lambda.
\]
Let us denote
\[
\Omega=\bigcup_{j}Q_{j}
\]
As usual,  we write $f=g+h$ where $g$,
the ``good'' part of $f$, is defined as 
\[
g(x)=\begin{cases}
f(x) & \quad x\in\Omega^{c}\\
f_{Q_{j}} & \quad x\in Q_{j}
\end{cases}
\]
and verifies that $|g(x)|\leq2^{n}\lambda$ a.e. and $h=\sum h_{j}$
where $h_{j}=\left(f-f_{Q_{j}}\right)\chi_{Q_{j}}$ and $f_{Q_{j}}=\frac{1}{|Q_{j}|}\int_{Q_{j}}f(x)dx$.
We denote $w^{*}(x)=w(x)\chi_{\mathbb{R}^{n}\setminus\tilde{\Omega}}(x)$
and $w_{j}(x)=w(x)\chi_{\mathbb{R}^{n}\setminus\tilde{Q_j}}$ where
$\tilde{Q_{j}}=5\sqrt{n}Q_{j}$ and $\tilde{\Omega}=\bigcup_{j}\tilde{Q}_{j}$.
Using that decomposition we can write 
\[
\begin{split}w\left(\left\{ x\in\mathbb{R}^{n}\,:\,|[b,T]f(x)|>\lambda\right\} \right) & \leq w\left(\left\{ x\in\mathbb{R}^{n}\setminus\tilde{\Omega}\,:\,|[b,T]g(x)|>\frac{\lambda}{2}\right\} \right)+w(\tilde{\Omega})\\
 & +w\left(\left\{ x\in\mathbb{R}^{n}\setminus\tilde{\Omega}\,:\,|[b,T]h(x)|>\frac{\lambda}{2}\right\} \right)\\
 & =I+II+III
\end{split}
\]
To end the proof we have to estimate $I,II$ and $III$. Let us begin with $I$. If $p>0$, Chebyschev's
inequality gives
\[
w\left(\left\{ x\in\mathbb{R}^{n}\setminus\tilde{\Omega}\,:\,|[b,T]g(x)|>\frac{\lambda}{2}\right\} \right)\leq\frac{2^{p}}{\lambda^{p}}\int_{\mathbb{R}^n}|[b,T]g(x)|^{p}w^{*}(x)dx.
\]
Let us choose $1+\frac{\varepsilon}{3\left(1+\frac{1}{s}\right)}<p<1+\frac{\varepsilon}{2\left(1+\frac{1}{s}\right)}$
y $\delta=\varepsilon-\left(1+\frac{1}{s}\right)\left(p-1\right)$.
For that choice of $p$ and $\delta$, is easy to check that 
\[
\left(p'\right)^{2p}p^{\left(1+\frac{1}{s}\right)p}\left(\frac{p-1}{\delta}\right)^{\frac{p}{p'}}\leq c_{s}\frac{1}{\varepsilon^{2}} 
\quad \text{and} \quad \left(1+\frac{1}{s}\right)p-1+\delta = \frac{1}{s}+\varepsilon.
\]

Using now Theorem \ref{TipoFuerte}, we have that 
\[
\begin{split} & \frac{2^{p}}{\lambda^{p}}\int_{\mathbb{R}^{n} }|[b,T]g(x)|^{p}w^{*}(x)dx\\
 & \leq c\,\left(p'\right)^{2p}p^{\left(1+\frac{1}{s}\right)p}\left(\frac{p-1}{\delta}\right)^{\frac{p}{p'}}
\int_{\mathbb{R}^{n}}|g(x)|^{p}M_{L(\log L)^{\left(1+\frac{1}{s}\right)p-1+\delta}}w^{*}(x)dx\\
 & \leq c\, \frac{1}{\varepsilon^{2}}\frac{2^{p}}{\lambda^{p}}\int_{\mathbb{R}^{n}}|g(x)|^{p}M_{L(\log L)^{\frac{1}{s}+\varepsilon}}w^{*}(x)dx\leq
c\,\frac{1}{\varepsilon^{2}}\frac{1}{\lambda}\int_{\mathbb{R}^{n}}|g(x)|M_{L(\log L)^{\frac{1}{s}+\varepsilon}}w^{*}(x)dx\\
 & \leq c\, \frac{1}{\varepsilon^{2}}\frac{1}{\lambda}\left(\int_{\mathbb{R}^{n}\setminus\Omega}|f(x)|M_{L(\log L)^{\frac{1}{s}+\varepsilon}}w(x)dx+\int_{ \Omega }|g(x)|M_{L(\log L)^{\frac{1}{s}+\varepsilon}}w^{*}(x)dx\right)
\end{split}
\]
and it suffices to estimate last integral. Indeed, 
\[
\begin{split} & \int_{\Omega }|g(x)|M_{L(\log L)^{\frac{1}{s}+\varepsilon}}w^{*}(x)dx\leq\sum_{j} |f|_{Q_{j}}\,\int_{Q_{j}}M_{L(\log L)^{\frac{1}{s}+\varepsilon}}w_{j}(x)dx\\
& \leq c\,\sum_{j} |Q_{j}| \frac{1}{|Q_{j}|}\int_{Q_{j}}|f(y)|dy\inf_{z\in Q_{j}}M_{L(\log L)^{\frac{1}{s}+\varepsilon}}w_{j}(z)\\
 & =c\sum_{j}\int_{Q_{j}}|f(y)|\inf_{z\in Q_{j}}M_{L(\log L)^{\frac{1}{s}+\varepsilon}}w_{j}(z)dy\leq c\sum_{j}\int_{Q_{j}}|f(y)|M_{L(\log L)^{\frac{1}{s}+\varepsilon}}w_{j}(y)dy\\
 & \leq c\int_{\Omega}|f(y)|M_{L(\log L)^{\frac{1}{s}+\varepsilon}}w(y)dy.
\end{split}
\]
Summarizing, we obtain that 
\[
I\leq c\frac{1}{\varepsilon^{2}}\int_{\mathbb{R}^{n}}\frac{|f(y)|}{\lambda}M_{L(\log L)^{\frac{1}{s}+\varepsilon}}w(y)dy.
\]
For $II$ we have the following standard estimate
\[
\begin{split}II & =w(\tilde{\Omega})\leq\sum_{j}\int_{5\sqrt{n}Q_{j}}w(x)dx=\sum_{j}|5\sqrt{n}Q_{j}|\frac{1}{|5\sqrt{n}Q_{j}|}\int_{5\sqrt{n}Q_{j}}w(x)dx\\
 & \leq\sum_{j}\left(5\sqrt{n}\right)^{n}|Q_{j}|\inf_{z\in Q_{j}}Mw(z)\leq\left(5\sqrt{n}\right)^{n}\sum_{j}\frac{1}{\lambda}\int_{Q_{j}}f(y)dy\inf_{z\in Q_{j}}Mw(z)\\
 & \leq\left(5\sqrt{n}\right)^{n}\sum_{j}\frac{1}{\lambda}\int_{Q_{j}}Mw(y)f(y)dy\leq\left(5\sqrt{n}\right)^{n}\int_{\mathbb{R}^{n}}\frac{f(y)}{\lambda}Mw(y)dy
\end{split}
\]
To estimate $III$ we split the operator as follows 
\[
[b,T]h=\sum_{j}[b,T]h_{j}=\sum_{j}\left(bT(h_{j})-T(bh_{j})\right)=\sum_{j}\left(b-b_{Q_{j}}\right)T(h_{j})-\sum_{j}T\left(\left(b-b_{Q_{j}}\right)h_{j}\right).
\]
Then we continue with 
\[
\begin{split}III & \leq w\left(\left\{ x\in\mathbb{R}^{n}\setminus\tilde{\Omega}\,:\,\left|\sum_{j}\left(b(x)-b_{Q_{j}}\right)Th_{j}(x)\right|>\frac{\lambda}{4}\right\} \right)\\
 & +w\left(\left\{ x\in\mathbb{R}^{n}\setminus\tilde{\Omega}\,:\,\left|\sum_{j}T\left(\left[b-b_{Q_{j}}\right]h_{j}\right)(x)\right|>\frac{\lambda}{4}\right\} \right)\\
 & =A+B
\end{split}
\]
To estimate $A$ we use standard computations based on the smoothness property  of the kernel $K$ and the
cancellation of each $h_{j}$, 
\[
\begin{split}A & \leq\frac{c}{\lambda}\int_{\mathbb{R}^{n}\setminus\tilde{\Omega}}\sum_{j}|b(x)-b_{Q_{j}}|\left|Th_{j}(x)\right|w(x)dx\\
 & \leq\frac{c}{\lambda}\sum_{j}\int_{\mathbb{R}^{n}\setminus\tilde{Q_{j}}}|b(x)-b|w(x)\int_{Q_{j}}|h_{j}(y)|\left|K(x,y)-K(x,x_{Q_{j}})\right|dydx\\
 & \leq\frac{c}{\lambda}\sum_{j}\int_{Q_{j}}|h_{j}(y)|\int_{\mathbb{R}^{n}\setminus\tilde{Q_{j}}}|K(x,y)-K(x,x_{Q_{j}})||b(x)-b_{Q_{j}}|w_{j}(x)dxdy\\
 & \leq\frac{c}{\lambda}\sum_{j}\int_{Q_{j}}|h_{j}(y)|\int_{\mathbb{R}^{n}\setminus\tilde{Q_{j}}}\frac{\left|y-x_{Q_{j}}\right|^{\gamma}}{\left|x-x_{Q_{j}}\right|^{n+\gamma}}|b(x)-b_{Q_{j}}|w_{j}(x)dxdy\\
 & \leq\frac{c}{\lambda}\sum_{j}\int_{Q_{j}}|h_{j}(y)|\sum_{k=1}^{\infty}\int_{2^{k}l(Q_{j})\leq|x-x_{Q_{j}}|<2^{k+1}l(Q_j)}\frac{\left|y-x_{Q_{j}}\right|^{\gamma}}{\left|x-x_{Q_{j}}\right|^{n+\gamma}}|b(x)-b_{Q_{j}}|w_{j}(x)dxdy\\
 & \leq\frac{c}{\lambda}\sum_{j}\left(\int_{Q_{j}}|h_{j}(y)|dy\right)\sum_{k=1}^{\infty}\frac{2^{-\gamma k}}{\left|2^{k+1}Q_{j}\right|}\int_{2^{k+1}Q_{j}}|b(x)-b_{Q_{j}}|w_{j}(x)dx
\end{split}
\]
We now fix one term of the sum. Using generalized Hölder inequality,
Lemma \ref{Lemma:HolderGeneralizado}, we have 
\[
\begin{split} & \sum_{k=0}^{\infty}\frac{2^{-\gamma k}}{\left|2^{k+1}Q_{j}\right|}\int_{2^{k+1}Q_{j}}|b(x)-b_{Q_{j}}|w_{j}(x)dx\\
 & \leq\sum_{k=0}^{\infty}\frac{2^{-\gamma k}}{\left|2^{k+1}Q_{j}\right|}\int_{2^{k+1}Q_{j}}|b(x)-b_{2^{k+1}Q_{j}}|w_{j}(x)dx\\
 & +\sum_{k=0}^{\infty}\frac{2^{-\gamma k}}{\left|2^{k+1}Q_{j}\right|}\int_{2^{k+1}Q_{j}}|b_{2^{k+1}Q_{j}}-b_{Q_{j}}|w_{j}(x)dx\\
 & \leq\sum_{k=1}^{\infty}2^{-\gamma k}\|b-b_{2^{k+1}Q_{j}}\|_{\exp L^{s},2^{k+1}Q_{j}}\|w_{j}\|_{L\log L^{\frac{1}{s}},2^{k+1}Q_{j}}\\
 & +\sum_{k=1}^{\infty}2^{-\gamma k}(k+1)\|b\|_{Osc_{expL^{s}}}\inf_{z\in Q_{j}}Mw_{j}(z)\\
 & \leq\sum_{k=1}^{\infty}2^{-\gamma k}\|b\|_{Osc_{expL^{s}}}\inf_{z\in Q_{j}}M_{L\log L^{\frac{1}{s}}}w_{j}(z)\\
 & +\sum_{k=1}^{\infty}2^{-\gamma k}(k+1)\|b\|_{Osc_{expL^{s}}}\inf_{z\in Q_{j}}Mw_{j}(z)\\
 & \leq c\left(\inf_{z\in Q_{j}}M_{L\log L^{\frac{1}{s}}}w_{j}(z)\sum_{k=1}^{\infty}2^{-\gamma k}+\inf_{z\in Q_{j}}Mw_{j}(z)\sum_{k=1}^{\infty}2^{-\gamma k}(k+1)\right)\\
 & \leq c\inf_{z\in Q_{j}}M_{L\log L^{\frac{1}{s}}}w_{j}(z)
\end{split}
\]
 Consequently, 
\[
\begin{split}A & \leq\frac{c}{\lambda}\sum_{j} \int_{Q_{j}}|h_{j}(y)|dy \inf_{y\in Q_{j}}M_{L\log L^{\frac{1}{s}}}(w_{j})(y)\\
 & \leq\frac{c}{\lambda}\sum_{j}\int_{Q_{j}}M_{L\log L^{\frac{1}{s}}}(w_{j})(y)|h_{j}(y)|dy\\
 & \leq\frac{c}{\lambda}\left(\int_{\mathbb{R}^{n}}|f(y)|M_{L\log L^{\frac{1}{s}}}(w_{j})(y)dy+\sum_{j}\int_{Q_{j}}M_{L\log L^{\frac{1}{s}}}(w_{j})(y)|f_{Q_{j}}|dy\right)\\
 & \leq\frac{c}{\lambda}\left(\int_{\mathbb{R}^{n}}|f(y)|M_{L\log L^{\frac{1}{s}}}(w_{j})(y)dy+\sum_{j}\int_{Q_{j}}f(y)dy\inf_{z\in Q_{j}}M_{L\log L^{\frac{1}{s}}}(w_{j})(z)dy\right)\\
 & \leq\frac{c}{\lambda}\int_{\mathbb{R}^{n}}|f(y)|M_{L\log L^{\frac{1}{s}}}(w_{j})(y)dy
\end{split}
\]
To end the proof we estimate $B$. Theorem 1.1 from \cite{MR3327006}
gives
\[
\begin{split}B & =w^{*}\left(\left\{ x\in\mathbb{R}^{n}\,:\,\left|\sum_{j}T\left(\left[b-b_{Q_{j}}\right]h_{j}\right)(x)\right|>\frac{\lambda}{4}\right\} \right)\\
 & \leq c\frac{1}{\varepsilon}\frac{1}{\lambda}\int_{\mathbb{R}^{n}}\left|\sum_{j}\left(b(x)-b_{Q_{j}}\right)h_{j}\right|(x)M_{L(\log L)^{\varepsilon}}(w^{*})(x)dx\\
 & \leq c\frac{1}{\varepsilon}\frac{1}{\lambda}\sum_{j}\int_{Q_{j}}\left|b(x)-b_{Q_{j}}\right|\left|f(x)-f_{Q_{j}}\right|M_{L(\log L)^{\varepsilon}}(w_{j})(x)dx\\
 & \leq c\frac{1}{\varepsilon}\frac{1}{\lambda}\sum_{j}\inf_{z\in Q_{j}}M_{L(\log L)^{\varepsilon}}(w_{j})(z)\left(\int_{Q_{j}}\left|b(x)-b_{Q_{j}}\right|\left|f(x)\right|dx+\int_{Q_{j}}\left|b(x)-b_{Q_{j}}\right|\left|f_{Q_{j}}\right|dx\right)\\
 & =\frac{1}{\varepsilon}\left(B_{1}+B_{2}\right)
\end{split}
\]
For $B_{2}$
\[
\begin{split}B_{2} & =\frac{c}{\lambda}\sum_{j}\inf_{z\in Q_{j}}M_{L(\log L)^{\varepsilon}}(w_{j})(z)\int_{Q_{j}}\left|b(x)-b_{Q_{j}}\right|\left|f_{Q_{j}}\right|dx\\
 & \leq\frac{c}{\lambda}\sum_{j}\frac{1}{|Q_{j}|}\int_{Q_{j}}\left|b(x)-b_{Q_{j}}\right|dx\int_{Q_{j}}|f(y)|M_{L(\log L)^{\varepsilon}}(w_{j})(y)dx\\
 & \leq\frac{c}{\lambda}\sum_{j}\|b\|_{Osc_{expL^{s}}}\int_{Q_{j}}|f(y)|M_{L(\log L)^{\varepsilon}}(w_{j})(y)dx\\
 & \leq c\sum_{j}\int_{Q_{j}}\frac{|f(y)|}{\lambda}M_{L(\log L)^{\varepsilon}}(w_{j})(y)dy\\
 & \leq c\int_{\mathbb{R}^{n}}\frac{|f(x)|}{\lambda}M_{L(\log L)^{\varepsilon}}w(x)dx.
\end{split}
\]
For $B_{1}$ we use the generalized Hölder inequality Lemma \ref{Lemma:HolderGeneralizado} and we obtain 
\begin{equation}\label{HGB1}
\begin{split}B_{1} & =\frac{c}{\lambda}\sum_{j}\inf_{z\in Q_{j}}M_{L(\log L)^{\varepsilon}}(w_{j})(z)\int_{Q_{j}}\left|b(x)-b_{Q_{j}}\right|\left|f(x)\right|dx\\
 & \leq c\sum_{j}\inf_{z\in Q_{j}}M_{L(\log L)^{\varepsilon}}(w_{j})(z)\frac{1}{\lambda}\left|Q_{j}\right|\|b\|_{Osc_{expL^{s}}}\|f\|_{L(\log L)^\frac{1}{s} L,Q_{j}}\\
 & =c\sum_{j}\inf_{z\in Q_{j}}M_{L(\log L)^{\varepsilon}}(w_{j})(z)\frac{1}{\lambda}\left|Q_{j}\right|\|f\|_{L(\log L)^\frac{1}{s},Q_{j}}.
\end{split}
\end{equation}
Now we see that 
\begin{equation}\label{CZB1}
\begin{split}\frac{1}{\lambda}\left|Q_{j}\right|\|f\|_{L(\log L)^\frac{1}{s},Q_{j}} & \leq\frac{1}{\lambda}|Q_{j}|\inf_{\mu>0}\left\{ \mu+\frac{\mu}{|Q_{j}|}\int_{Q_{j}}\Phi_\frac{1}{s}\left(\frac{|f(x)|}{\mu}\right)dx\right\} \\
 & \leq\frac{1}{\lambda}|Q_{j}|\left(\lambda+\frac{\lambda}{|Q_{j}|}\int_{Q_{j}}\Phi_\frac{1}{s}\left(\frac{|f(x)|}{\lambda}\right)dx\right)=|Q_{j}|+\int_{Q_{j}}\Phi_\frac{1}{s}\left(\frac{|f(x)|}{\lambda}\right)dx\\
 & \leq\frac{1}{\lambda}\int_{Q_{j}}|f(x)|dx+\int_{Q_{j}}\Phi_\frac{1}{s}\left(\frac{|f(x)|}{\lambda}\right)dx\leq 2\int_{Q_{j}}\Phi_\frac{1}{s}\left(\frac{|f(x)|}{\lambda}\right)dx.
\end{split}
\end{equation}
Consequently 
\[
\begin{split}B_{1} & \leq c\sum_{j}\inf_{z\in Q_{j}}M_{L(\log L)^{\varepsilon}}(w_{j})(z)\int_{Q_{j}}\Phi_\frac{1}{s}\left(\frac{|f(x)|}{\lambda}\right)dx\\
 & \leq c\sum_{j}\int_{Q_{j}}\Phi_\frac{1}{s}\left(\frac{|f(x)|}{\lambda}\right)M_{L(\log L)^{\varepsilon}}(w_{j})(x)dx\\
 & \leq c\int_{\mathbb{R}^{n}}\Phi_\frac{1}{s}\left(\frac{|f(x)|}{\lambda}\right)M_{L(\log L)^{\varepsilon}}(w)(x)dx.
\end{split}
\]

\end{proof}

\subsection{Case $k>1$}
\begin{proof}
Let us suppose that the desired inequality holds for $l\leq k-1$ 
symbols. By homogeneity we may assume that $\|b\|_{Osc_{expL^{s_{1}}}}=\dots=\|b\|_{Osc_{expL^{s_{k}}}}=1$.
Using the Calderón-Zygmund decomposition with the same notation used in the case $k=1$ we can write 
\[
\begin{split}w\left(\left\{ x\in\mathbb{R}^{n}\,:\,|T_{\vec{b}}f(x)|>\lambda\right\} \right) & \leq w\left(\left\{ x\in\mathbb{R}^{n}\setminus\tilde{\Omega}\,:\,|T_{\vec{b}}g(x)|>\frac{\lambda}{2}\right\} \right)+w(\tilde{\Omega})\\
 & +w\left(\left\{ x\in\mathbb{R}^{n}\setminus\tilde{\Omega}\,:\,|T_{\vec{b}}h(x)|>\frac{\lambda}{2}\right\} \right)\\
 & =I+II+III
\end{split}
\]
We consider now each term separately. To estimate $I$ we use Chebyschev's
inequality for $p>1$ that will be chosen appropriately,
\[
w\left(\left\{ x\in\mathbb{R}^{n}\setminus\tilde{\Omega}\,:\,|T_{\vec{b}}g(x)|>\frac{\lambda}{2}\right\} \right)\leq\frac{2^{p}}{\lambda^{p}}\int_{\mathbb{R}}|T_{\vec{b}}g(x)|^{p}w^{*}(x)dx.
\]
Let us choose, as we did in the case $k=1$, $p$ such that $1+\frac{\varepsilon}{3\left(1+\frac{1}{s}\right)}<p<1+\frac{\varepsilon}{\left(1+\frac{1}{s}\right)2}$ 
and $\delta=\varepsilon-\left(1+\frac{1}{s}\right)p$. For this choice
of $p$ and $\delta$ we have that
\[
\left(p'\right)^{(k+1)p}p^{\left(1+\frac{1}{s}\right)p}\left(\frac{p-1}{\delta}\right)^{\frac{1}{p'}}\leq c_{s}\frac{1}{\varepsilon^{k+1}} 
\quad \text{and} \quad \left(1+\frac{1}{s}\right)p-1+\delta = \frac{1}{s}+\varepsilon
\]

Using now theorem \ref{TipoFuerte} and the choice of $\delta$ and $p$ we have that 
\[
\begin{split}\frac{2^{p}}{\lambda^{p}}\int_{\mathbb{R}}|T_{\vec{b}}g(x)|^{p}w^{*}(x)dx & \leq c_{n}\left(p'\right)^{(k+1)p}p^{\left(1+\frac{1}{s}\right)p}\left(\frac{p-1}{\delta}\right)^{\frac{1}{p'}}\int_{\mathbb{R}}|g(x)|^{p}M_{L(\log L)^{\left(1+\frac{1}{s}\right)p-1+\delta}}w^{*}(x)dx\\
 & \leq c\frac{1}{\varepsilon^{k+1}}\frac{2^{p}}{\lambda^{p}}\int_{\mathbb{R}}|g(x)|^{p}M_{L(\log L)^{\frac{1}{s}+\varepsilon}}w^{*}(x)dx
\end{split}
\]
Arguing as in the case $k=1$ we obtain that 
\[
I\leq c_{n}\frac{1}{\varepsilon^{k+1}}\int_{\mathbb{R}^{n}}\frac{|f(y)|}{\lambda}M_{L(\log L)^{\frac{1}{s}+\varepsilon}}w(y)dy.
\]
For $II$, as in the case $k=1$, we have the following estimate
\[
II\leq3^{n}\int_{\mathbb{R}^{n}}\frac{f(y)}{\lambda}Mw(y)dy
\]
 It remains to estimate  $III$. Following the
computations of page 684 of \cite{MR1895740} we can write
\begin{equation}
\label{Tvecb}
\begin{split}T_{\vec{b}}f(x) & =(b_{1}(x)-\lambda_{1})\dots(b_{k}(x)-\lambda_{k})Tf(x)\\
 & +(-1)^{k}T\left((b_{1}-\lambda_{1})\dots(b_{k}-\lambda_{k})f\right)(x)\\
 & +\sum_{i=1}^{k-1}\sum_{\sigma\in C_{i}(b)}(-1)^{k-i}\left(b(x)-\vec{\lambda}\right)_{\sigma}\int_{\mathbb{R}^{n}}\left(b(y)-\vec{\lambda}\right)_{\sigma'}K(x,y)f(y)dy.
\end{split}
\end{equation}	
Now we work on the last double summation. We observe that for each term we can write
\[
\begin{split} & \left(b(x)-\vec{\lambda}\right)_{\sigma}\int_{\mathbb{R}^{n}}\left(b(y)-\vec{\lambda}\right)_{\sigma'}K(x,y)f(y)dy\\
 & =\int_{\mathbb{R}^{n}}\left(b(y)-\vec{\lambda}\right)_{\sigma'}\left(\left[b(x)-b(y)\right]+\left[b(y)-\vec{\lambda}\right]\right)_{\sigma}K(x,y)f(y)dy\\
 & \stackrel{\tau\cup\tau'=\sigma}{=}\int_{\mathbb{R}^{n}}\left(b(y)-\vec{\lambda}\right)_{\sigma'}\sum_{j=0}^{\#\sigma}\sum_{\tau\in C_{j}(\sigma)}\left(b(x)-b(y)\right)_{\tau}\left(b(y)-\vec{\lambda}\right)_{\tau'}K(x,y)f(y)dy\\
 & =\sum_{j=0}^{\#\sigma}\sum_{\tau\in C_{j}(\sigma)}\int_{\mathbb{R}^{n}}\left(b(x)-b(y)\right)_{\tau}\left(b(y)-\vec{\lambda}\right)_{\sigma'\cup\tau'}K(x,y)f(y)dy\\
 & =\sum_{j=0}^{\#\sigma}\sum_{\tau\in C_{j}(\sigma)}T_{\vec{\tau}}\left(\left(b-\vec{\lambda}\right)_{\sigma'\cup\tau'}f\right)\\
 & =T\left((b_{1}-\lambda_{1})\dots(b_{k}-\lambda_{k})f\right)(x)+\sum_{j=1}^{\#\sigma}\sum_{\tau\in C_{j}(\sigma)}T_{\vec{\tau}}\left(\left(b-\vec{\lambda}\right)_{\sigma'\cup\tau'}f\right).
\end{split}
\]
Plugging this into the double summation of (\ref{Tvecb}), since $\tau\cup \tau'\cup \sigma' = b$ we can write, 
\[
\begin{split} & \sum_{i=1}^{k-1}\sum_{\sigma\in C_{i}(b)}(-1)^{k-i}\left(b(x)-\vec{\lambda}\right)_{\sigma}\int_{\mathbb{R}^{n}}\left(b(y)-\vec{\lambda}\right)_{\sigma'}K(x,y)f(y)dy\\
 & =c_{k}T\left((b_{1}-\lambda_{1})\dots(b_{k}-\lambda_{k})f\right)(x)+\sum_{i=1}^{k-1}\sum_{\sigma\in C_{i}(b)}c_{\sigma}T_{\vec{\sigma}}\left(\left(b-\vec{\lambda}\right)_{\sigma'}f\right)
\end{split}
\]
where $c_{\sigma}$ is a constant that counts the number of repetitions
of each $T_{\vec{\sigma}}$. Summarizing 
\[
\begin{split}T_{\vec{b}}f(x) & =(b_{1}(x)-\lambda_{1})\dots(b_{k}(x)-\lambda_{k})Tf(x)\\
 & +c_{k}T\left((b_{1}-\lambda_{1})\dots(b_{k}-\lambda_{k})f\right)(x)\\
 & +\sum_{i=1}^{k-1}\sum_{\sigma\in C_{i}(b)}c_{\sigma}T_{\vec{\sigma}}\left(\left(b(y)-\vec{\lambda}\right)_{\sigma'}f\right)(x)
\end{split}
\]
Using this for each $h_{j}$ and summing on $j$, 

\[
\begin{split}\sum_{j}T_{\vec{b}}h_{j}(x) & =\sum_{j}(b_{1}(x)-\lambda_{1})\dots(b_{k}(x)-\lambda_{k})Th_{j}(x)\\
 & +\sum_{j}c_{k}T\left((b_{1}-\lambda_{1})\dots(b_{k}-\lambda_{k})h_{j}\right)(x)\\
 & +\sum_{j}\sum_{i=1}^{k-1}\sum_{\sigma\in C_{i}(b)}c_{\sigma}T_{\vec{\sigma}}\left(\left(b-\vec{\lambda}\right)_{\sigma'}h_{j}\right)(x)
\end{split}
\]
Then we can estimate $III$ as follows
\[
\begin{split}III & \leq w\left(\left\{ y\in\mathbb{R}^{n}\setminus\tilde{\Omega}\,:\,\left|\sum_{j}\left(b_{1}(x)-\left(b_{1}\right)_{Q_{j}}\right)\dots\left(b_{k}(x)-\left(b_{k}\right)_{Q_{j}}\right)Th_{j}(x)\right|>\frac{\lambda}{6}\right\} \right)\\
 & +w\left(\left\{ y\in\mathbb{R}^{n}\setminus\tilde{\Omega}\,:\,\left|\sum_{j}c_{k}T\left(\left(b_{1}-\left(b_{1}\right)_{Q_{j}}\right)\dots\left(b_{k}-\left(b_{k}\right)_{Q_{j}}\right)h_{j}\right)(x)\right|>\frac{\lambda}{6}\right\} \right)\\
 & +w\left(\left\{ y\in\mathbb{R}^{n}\setminus\tilde{\Omega}\,:\,\left|\sum_{j}\sum_{i=1}^{k-1}\sum_{\sigma\in C_{i}^{k}}c_{\sigma}T_{\vec{\sigma}}\left(\left(b-\vec{b_{Q_{j}}}\right)_{\sigma'}h_{j}\right)(x)\right|>\frac{\lambda}{6}\right\} \right)\\
 & =L_{1}+L_{2}+L_{3}
\end{split}
\]
To estimate $L_{1}$ we denote $w_{j}=\chi_{\mathbb{R}^{n}\setminus 5\sqrt{n} Q_{j}}w$
and $B(x)=\prod_{i=1}^{k}\left|b_{i}(x)-\left(b_{i}\right)_{Q_{j}}\right|$.
Then
\[
\begin{split}L_{1} & \leq\frac{c}{\lambda}\int_{\mathbb{R}^{n}\setminus\tilde{\Omega}}\left|\sum_{j}\left(b_{1}(x)-\left(b_{1}\right)_{Q_{j}}\right)\dots\left(b_{k}(x)-\left(b_{k}\right)_{Q_{j}}\right)Th_{j}(x)\right|w(x)dx\\
 & \leq\sum_{j}\frac{c}{\lambda}\int_{\mathbb{R}^{n}\setminus\tilde{\Omega}}B(x)\left|Th_{j}(x)\right|w(x)dx=\\
 & \leq\sum_{j}\frac{c}{\lambda}\int_{\mathbb{R}^{n}\setminus\tilde{\Omega}}B(x)w(x)\left(\int_{Q_{j}}|h_{j}(y)||K(x,y)-K(x,x_{Q_{j}})|dy\right)dx\\
 & \leq\sum_{j}\frac{c}{\lambda}\int_{Q_{j}}|h_{j}(y)|\int_{\mathbb{R}^{n}\setminus5\sqrt{n}Q_{j}}B(x)w_{j}(x)|K(x,y)-K(x,x_{Q_{j}})|dxdy\\\end{split}
\]

A standard computation using the smoothness condition of $K$ yields that the latter is bounded by
\begin{equation}
\label{CompK1}
\begin{split} & \sum_{j}\frac{c}{\lambda}\int_{Q_{j}}|h_{j}(y)|\sum_{m}\int_{2^{m}l(Q_j)\leq|x-x_{Q_{j}}|\leq2^{m+1}l(Q_j)}B(x)w_{j}(x)\frac{|y-x_{Q_{j}}|^{\gamma}}{|x-x_{Q_{j}}|^{n+\gamma}}dxdy\\
 & \leq\sum_{j}\frac{c}{\lambda}\int_{Q_{j}}|h_{j}(y)|\sum_{m}\frac{2^{-m\gamma}}{(2^{m+1}l(Q_j))^{n}}\int_{|x-x_{Q_{j}}|\leq2^{m+1}l(Q_j)}B(x)w_{j}(x)dxdy=
\end{split}
\end{equation}
Let us estimate the inner sum. We have that
\[
\begin{split} & \sum_{m}\frac{2^{-m\gamma}}{(2^{m+1}l(Q_j))^{n}}\int_{|x-x_{Q_{j}}|\leq2^{m+1}l(Q_j)}B(x)w_{j}(x)dx\\
 & \leq\sum_{m}\frac{2^{-m\gamma}}{(2^{m+1}l(Q_j))^{n}}\int_{2^{m+1}Q_{j}}\prod_{i=1}^{k}\left|b_{i}(x)-\left(b_{i}\right)_{Q_{j}}\right|w_{j}(x)dx\\
 & =\sum_{m}\frac{2^{-m\gamma}}{(2^{m+1}l(Q_j))^{n}}\int_{2^{m+1}Q_{j}}\prod_{i=1}^{k}\left(\left|b_{i}(x)-\left(b_{i}\right)_{2^{m+1}Q_{j}}\right|+\left|\left(b_{i}\right)_{2^{m+1}Q_{j}}-\left(b_{i}\right)_{Q_{j}}\right|\right)w_{j}(x)dx\\
 & =\sum_{m}\frac{2^{-m\gamma}}{(2^{m+1}l(Q_j))^{n}}\int_{2^{m+1}Q_{j}}\sum_{l=0}^{k}\sum_{\sigma\in C_{l}^{k}}\left(\prod_{i\in\sigma}\left|b_{i}(x)-\left(b_{i}\right)_{2^{m+1}Q_{j}}\right|\right)\left(\prod_{i\in\sigma'}\left|\left(b_{i}\right)_{2^{m+1}Q_{j}}-\left(b_{i}\right)_{Q_{j}}\right|\right)w_{j}(x)dx\\
 & =\sum_{m}\sum_{l=0}^{k}\sum_{\sigma\in C_{l}(b)}\left(\prod_{i\in\sigma'}\left|\left(b_{i}\right)_{2^{m+1}Q_{j}}-\left(b_{i}\right)_{Q_{j}}\right|\right)\frac{2^{-m\gamma}}{(2^{m+1}l(Q_j))^{n}}\int_{2^{m+1}Q_{j}}\left(\prod_{i\in\sigma}\left|b_{i}(x)-\left(b_{i}\right)_{2^{m+1}Q_{j}}\right|\right)w_{j}(x)dx\\
 & \leq\sum_{m}\sum_{l=0}^{k}\sum_{\sigma\in C_{l}(b)}\left(\prod_{i\in\sigma'}\|b_{i}\|_{Osc_{expL^{s_{i}}}}\right)\frac{2^{-m\gamma}}{(2^{m+1}l(Q_j))^{n}}\int_{2^{m+1}Q_{j}}\left(\prod_{i\in\sigma}\left|b_{i}(x)-\left(b_{i}\right)_{2^{m+1}Q_{j}}\right|\right)w_{j}(x)dx
\end{split}
\]
Applying Corollary \ref{CorolarioInv} we have that
\[
\begin{split} & \frac{1}{(2^{m+1}l(Q_j))^{n}}\int_{2^{m+1}Q_{j}}\left(\prod_{i\in\sigma}\left|b_{i}(x)-\left(b_{i}\right)_{2^{m+1}Q_{j}}\right|\right)w_{j}(x)dx\\
 & \leq c\left(\prod_{i\in\sigma}\|b_{i}\|_{Osc_{expL^{s_{i}}}}\right)\inf_{z\in2^{m+1}Q_{j}}M_{L(\log L)^{\sum_{i\in\sigma}\frac{1}{s_{i}}}}(w_{j})(x)
\end{split}
\]
Then for each $y\in Q_{j}$ 
\[
\begin{split} & \sum_{m}\sum_{l=0}^{k}\sum_{\sigma\in C_{l}(b)}\left(\prod_{i\in\sigma'}\|b_{i}\|_{Osc_{expL^{s_{i}}}}\right)\frac{2^{-m\gamma}}{(2^{m+1}l(Q_j))^{n}}\int_{2^{m+1}Q_{j}}\left(\prod_{i\in\sigma}\left|b_{i}(x)-\left(b_{i}\right)_{2^{m+1}Q_{j}}\right|\right)w_{j}(x)dx\\
 & \leq c\sum_{m}\frac{1}{2^{m\gamma}}\sum_{l=0}^{k}\sum_{\sigma\in C_{l}(b)}\left(\prod_{s\in\sigma'}\|b_{s}\|_{Osc_{expL^{\frac{1}{r_{s}}}}}\right)\left(\prod_{i\in\sigma}\|b_{i}\|_{Osc_{expL^{\frac{1}{r_{i}}}}}\right)\inf_{z\in2^{m+1}Q_{j}}M_{L(\log L)^{\sum_{i\in\sigma}\frac{1}{s_{i}}}}(w_{j})(x)\\
 & \leq c_{k}M_{L(\log L)^{\frac{1}{s}}}(w_{j})(y)\sum_{m}\frac{1}{2^{\gamma m}}=c_{k}M_{L(\log L)^{\frac{1}{s}}}(w_{j})(y).
\end{split}
\]
Continuing the computation in \eqref{CompK1} we have that by standard estimates,
\[
\begin{split}& \sum_{j}\frac{c}{\lambda}\int_{Q_{j}}|h_{j}(y)|\sum_{m}\frac{2^{-m\varepsilon}}{(2^{m+1}l(Q_j))^{n}}\int_{|x-x_{Q_{j}}|\leq2^{m+1}l(Q_j)}B(x)w_{j}(x)dxdy\\
 & \leq\frac{c_{k}}{\lambda}\sum_{j}\int_{Q_{j}}|h_{j}(y)|M_{L(\log L)^{\frac{1}{s}}}(w_{j})(y)dy\\
 & \leq\frac{c_{k}}{\lambda}\int_{\mathbb{R}^{n}}|f(y)|M_{L(\log L)^{\frac{1}{s}}}(w)(y)dy.
\end{split}
\]
Summarizing
\[
L_{1}\leq\frac{c_{k}}{\lambda}\int_{Q_{j}}|f(y)|M_{L(\log L)^{\frac{1}{s}}}(w)(y)dy.
\]
We shall work now on $L_{2}$. Theorem 1.1 from \cite{MR3327006}
gives 
\[
\begin{split}L_{2} & =\tilde{w}\left(\left\{ y\in\mathbb{R}^{n}\,:\,\left|c_{k}T\left(\sum_{j}\left(b_{1}-\left(b_{1}\right)_{Q_{j}}\right)\dots\left(b_{k}-\left(b_{k}\right)_{Q_{j}}\right)h_{j}\right)(x)\right|>\frac{\lambda}{6}\right\} \right)\\
 & \leq\frac{c}{\lambda}\frac{1}{\varepsilon}\int_{\mathbb{R}^{n}}\left|\sum_{j}\left[\left(b_{1}(x)-\left(b_{1}\right)_{Q_{j}}\right)\dots\left(b_{k}(x)-\left(b_{k}\right)_{Q_{j}}\right)h_{j}\right]\right|M_{L(\log L)^{\varepsilon}}\tilde{w}(x)dx\\
 & \leq\frac{c}{\lambda}\frac{1}{\varepsilon}\sum_{j}\int_{Q_{j}}B(x)\left|f(x)-f_{Q_{j}}\right|M_{L(\log L)^{\varepsilon}}w_{j}(x)dx\\
 & \leq\frac{c}{\lambda}\frac{1}{\varepsilon}\sum_{j}\inf_{z\in Q_{j}}M_{L(\log L)^{\varepsilon}}w_{j}(z)\left(\int_{Q_{j}}B(x)|f(x)|dx+\int_{Q_{j}}B(x)|f_{Q_{j}}|dx\right)\\
 & =\frac{c}{\lambda}\frac{1}{\varepsilon}\left(\sum_{j}\inf_{z\in Q_{j}}M_{L(\log L)^{\varepsilon}}w_{j}(z)\int_{Q_{j}}B(x)|f(x)|dx+\sum_{j}\inf_{z\in Q_{j}}M_{L(\log L)^{\varepsilon}}w_{j}(z)\int_{Q_{j}}B(x)|f_{Q_{j}}|dx\right)\\
 & =\frac{1}{\varepsilon}\left(L_{21}+L_{22}\right)
\end{split}
\]
We estimate first $L_{22}$ as follows
\[
\begin{split} & \frac{c}{\lambda}\sum_{j}\inf_{z\in Q_{j}}M_{L(\log L)^{\varepsilon}}w_{j}(z)\int_{Q_{j}}B(x)|f_{Q_{j}}|dx\\
 & =\frac{c}{\lambda}\sum_{j}\inf_{z\in Q_{j}}M_{L(\log L)^{\varepsilon}}w_{j}(z)\left(\frac{1}{|Q_{j}|}\int_{Q_{j}}B(x)dx\right)\left(\int_{Q_{j}}|f(x)|dx\right)
\end{split}
\]

Using Corrollary \ref{CorolarioInv} with $g=1$ and $f_i=\left|b_i-(b_i)_{Q_j}\right|$, we obtain the following estimate
\begin{equation}
\frac{1}{|Q_{j}|}\int_{Q_{j}}B(x)dx\leq c\prod_{i=1}^{m}\left\Vert b_{i}-\left(b_{i}\right)_{Q_{j}}\right\Vert _{\exp L^{s_{i}},Q_{j}}\leq c\|\vec{b}\|=c.\label{eq:invsExp}
\end{equation}
Then 
\[
\begin{split} & \frac{c}{\lambda}\sum_{j}\inf_{z\in Q_{j}}M_{L(\log L)^{\varepsilon}}w_{j}(z)\left(\frac{1}{|Q_{j}|}\int_{Q_{j}}B(x)dx\right)\left(\int_{Q_{j}}|f(x)|dx\right)\\
 & \leq\frac{c}{\lambda}\sum_{j}\inf_{z\in Q_{j}}M_{L(\log L)^{\varepsilon}}w_{j}(z)\left(\int_{Q_{j}}|f(x)|dx\right)\\
 & \leq\frac{c}{\lambda}\sum_{j}\int_{Q_{j}}|f(x)|M_{L(\log L)^{\varepsilon}}w_{j}(x)dx\\
 & \leq\frac{c}{\lambda}\int_{\mathbb{R}^{n}}|f(x)|M_{L(\log L)^{\varepsilon}}w_{j}(x)dx.
\end{split}
\]
Let us estimate now $L_{21}$. Using generalized Hölder inequality
(Lemma \ref{Lemma:HolderGeneralizado}) similarly as we did in \eqref{HGB1}
\[
\begin{split}L_{21} & =\frac{c}{\lambda}\sum_{j}\inf_{z\in Q_{j}}M_{L(\log L)^{\varepsilon}}w_{j}(z)\int_{Q_{j}}B(x)|f(x)|dx\\
 & \leq \frac{c}{\lambda}\sum_{j}\inf_{z\in Q_{j}}M_{L(\log L)^{\varepsilon}}w_{j}(z)|Q_{j}|\|f\|_{L(\log L)^{\frac{1}{s}},Q_{j}}
\end{split}
\]
since $\|\vec{b}\|=1$. Also the same computation used in \eqref{CZB1} based on properties of the Calderón-Zygmund cubes $Q_{j}$ yields
%
\[
\frac{1}{\lambda}|Q_{j}|\|f\|_{L(\log L)^{\frac{1}{s}},Q_{j}}  \leq 2\int_{Q_{j}}\Phi_{\frac{1}{s}}\left(\frac{|f(x)|}{\lambda}\right)dx 
\]
Hence
\[
\begin{split}L_{21} & \leq\frac{c}{\lambda}\sum_{j}\inf_{z\in Q_{j}}M_{L(\log L)^{\varepsilon}}w_{j}(z)|Q_{j}|\|f\|_{L(\log L)^{\frac{1}{s}},Q_{j}}\\
 & \leq c\sum_{j}\inf_{z\in Q_{j}}M_{L(\log L)^{\varepsilon}}w_{j}(z)2\int_{Q_{j}}\Phi_{\frac{1}{s}}\left(\frac{|f(x)|}{\lambda}\right)dx\\
 & \leq c\sum_{j}\int_{Q_{j}}\Phi_{\frac{1}{s}}\left(\frac{|f(x)|}{\lambda}\right)M_{L(\log L)^{\varepsilon}}w_{j}(x)dx\\
 & \leq c\int_{\mathbb{R}^{n}}\Phi_{\frac{1}{s}}\left(\frac{|f(x)|}{\lambda}\right)M_{L(\log L)^{\varepsilon}}w(x)dx.
\end{split}
\]
Putting $L_{21}$ and $L_{22}$ together we have that 
\[
L_{2}\leq c\frac{1}{\varepsilon}\int_{\mathbb{R}^{n}}\Phi_{\frac{1}{s}}\left(\frac{|f(x)|}{\lambda}\right)M_{L(\log L)^{\varepsilon}}w(x)dx
\]
To conclude the proof we are left with estimating $L_{3}$ as follows 
\[
\begin{split}L_{3} & =w\left(\left\{ y\in\mathbb{R}^{n}\setminus\tilde{\Omega}\,:\,\left|\sum_{i=1}^{k-1}\sum_{\sigma\in C_{i}(b)}c_{\sigma}T_{\vec{\sigma}}\left(\sum_{j}\left(b-\vec{\lambda}\right)_{\sigma'}h_{j}\right)(x)\right|>\frac{\lambda}{6}\right\} \right)\\
 & \leq w\left(\left\{ y\in\mathbb{R}^{n}\setminus\tilde{\Omega}\,:\,\left|\sum_{i=1}^{k-1}\sum_{\sigma\in C_{i}(b)}c_{\sigma}T_{\vec{\sigma}}\left(\sum_{j}\left(b-\vec{\lambda}\right)_{\sigma'}f\chi_{Q_{j}}\right)(x)\right|>\frac{\lambda}{12}\right\} \right)\\
 & +w\left(\left\{ y\in\mathbb{R}^{n}\setminus\tilde{\Omega}\,:\,\left|\sum_{i=1}^{k-1}\sum_{\sigma\in C_{i}(b)}c_{\sigma}T_{\vec{\sigma}}\left(\sum_{j}\left(b-\vec{\lambda}\right)_{\sigma'}f_{Q_{j}}\chi_{Q_{j}}\right)(x)\right|>\frac{\lambda}{12}\right\} \right)\\
 & =L_{31}+L_{32}
\end{split}
\]
To estimate $L_{31}$ we use the inductive hypothesis. 
\[
\begin{split} & L_{31}\\
 & =w\left(\left\{ y\in\mathbb{R}^{n}\setminus\tilde{\Omega}\,:\,\left|\sum_{i=1}^{k-1}\sum_{\sigma\in C_{i}(b)}c_{\sigma}T_{\vec{\sigma}}\left(\sum_{j}\left(b-\vec{\lambda}\right)_{\sigma'}f\chi_{Q_{j}}\right)(x)\right|>\frac{\lambda}{12}\right\} \right)\\
 & \leq c\sum_{i=1}^{k-1}\sum_{\sigma\in C_{i}(b)}w\left(\left\{ y\in\mathbb{R}^{n}\setminus\tilde{\Omega}\,:\,\left|T_{\vec{\sigma}}\left(\sum_{j}\left(b-\vec{\lambda}\right)_{\sigma'}f\chi_{Q_{j}}\right)(x)\right|>\frac{\lambda}{c_k}\right\} \right)\\
 & \leq c\sum_{i=1}^{k-1}\sum_{\sigma\in C_{i}(b)}\sum_{j}\frac{1}{\varepsilon^{\#\sigma+1}}\int_{Q_{j}}\Phi_{\sum_{i\in\sigma}\frac{1}{s_{i}}}\left( \|\vec{\sigma}\| \frac{|f(x)|}{\lambda}\left(b(x)-b_{Q_{j}}\right)_{\sigma'}\right)M_{L(\log L)^{\sum_{i\in\sigma}\frac{1}{s_{i}}+\varepsilon}}(w_{j})(x)dx
\end{split}
\]

Since we are assuming that $\|b_{1}\|_{Osc_{expL^{s_{1}}}}=\|b_{2}\|_{Osc_{expL^{s_{2}}}}=\dots=\|b_{k}\|_{Osc_{expL^{s_{k}}}}=1$, for each $\sigma \subseteq b$ we have that $\|\vec{\sigma}\|=1$. Then, 
\[
\begin{split} & c\sum_{i=1}^{k-1}\sum_{\sigma\in C_{i}(b)}\sum_{j}\frac{1}{\varepsilon^{\#\sigma+1}}\int_{Q_{j}}\Phi_{\sum_{i\in\sigma}\frac{1}{s_{i}}}\left(\frac{|f(x)|}{\lambda}\left(b(x)-b_{Q_{j}}\right)_{\sigma'}\right)M_{L(\log L)^{\sum_{i\in\sigma}\frac{1}{s_{i}}+\varepsilon}}(w_{j})(x)dx\\
 & \leq c\sum_{i=1}^{k-1}\sum_{\sigma\in C_{i}(b)}\sum_{j}\frac{1}{\varepsilon^{\#\sigma+1}}\inf_{z\in Q_{j}}M_{L(\log L)^{\sum_{i\in\sigma}\frac{1}{s_{i}}+\varepsilon}}(w_{j})(z)\int_{Q_{j}}\Phi_{\sum_{i\in\sigma}\frac{1}{s_{i}}}\left(\frac{|f(x)|}{\lambda}\left(b(x)-b_{Q_{j}}\right)_{\sigma'}\right)dx
\end{split}
\]
Let us consider now 
\[
\Phi_{u}^{-1}(t)=\frac{t}{\log(e+t)^{u}}\qquad\varphi_{v}^{-1}(t)=\log(1+t)^{\frac{1}{v}}.
\]
Then 
\[
\Phi_{\frac{1}{s}}^{-1}(t)\prod_{i\in\sigma}\varphi{s_i}^{-1}(t)=\frac{t}{\log(e+t)^{\sum\frac{1}{s_{i}}}}\prod_{i\in\sigma}\log(1+t)^{\frac{1}{s_{i}}}\leq\frac{t}{\log(e+t)^{\sum_{i\in\sigma'}\frac{1}{s_{i}}}}=\Phi_{\sum_{i\in\sigma'}\frac{1}{s_{i}}}^{-1}(t)
\]
and also we know that 
\[
\Phi_{u}(t)\simeq t\left(1+\log^{+}t\right)^{u},\varphi_{v}(t)=e^{t^{v}}-1.
\]
Taking that into account, Lemma \ref{Lema2.2Pt} gives
\begin{equation}
\begin{split} & \int_{Q_{j}}\Phi_{\sum_{i\in\sigma}\frac{1}{s_{i}}}\left(\frac{|f(x)|}{\lambda}\left(b(x)-b_{Q_{j}}\right)_{\sigma'}\right)dx\\
 & \leq\int_{Q_{j}}\Phi_{\frac{1}{s}}\left(\frac{|f(x)|}{\lambda}\right)dx+\sum_{i\in\sigma}\int_{Q_{j}}\left(\exp\left(\left|b_{i}(x)-\left(b_{i}\right)_{Q_{j}}\right|^{s_{i}}\right)-1\right)dx\\
 & \leq\int_{Q_{j}}\Phi_{\frac{1}{s}}\left(\frac{|f(x)|}{\lambda}\right)dx+c\sum_{i\in\sigma}|Q_{j}|\|b_{i}\|_{expL^{s_{i}},Q_{j}}\\
 & \leq\int_{Q_{j}}\Phi_{\frac{1}{s}}\left(\frac{|f(x)|}{\lambda}\right)dx+c\sum_{i\in\sigma}|Q_{j}|\|b_{i}\|_{Osc_{expL^{s_{i}}}}\qquad\left[\|b_{i}\|_{Osc_{expL^{s_{i}}}}=1\right]\\
 & \leq\int_{Q_{j}}\Phi_{\frac{1}{s}}\left(\frac{|f(x)|}{\lambda}\right)dx+ck|Q_{j}|\\
 & \leq c_k\int_{Q_{j}}\Phi_{\frac{1}{s}}\left(\frac{|f(x)|}{\lambda}\right)dx.
\end{split}
\label{eq:1}
\end{equation}
In the last step we used properties of the Calder\'on-Zygmund cubes. 

Plugging now that estimate, 
\[
\begin{split} & c\sum_{i=1}^{k-1}\sum_{\sigma\in C_{i}^{k}}\sum_{j}\frac{1}{\varepsilon^{\#\sigma+1}}\inf_{z\in Q_{j}}M_{L(\log L)^{\sum_{i\in\sigma'}\frac{1}{s_{i}}+\varepsilon}}(w_{j})(z)\int_{Q_{j}}\Phi_{\sum_{i\in\sigma'}\frac{1}{s_{i}}}\left(\frac{|f(x)|}{\lambda}\left(b(x)-b_{Q_{j}}\right)_{\sigma}\right)dx\\
 & \leq c_k\sum_{i=1}^{k-1}\sum_{\sigma\in C_{i}^{m}}\sum_{j}\frac{1}{\varepsilon^{\#\sigma+1}}\inf_{z\in Q_{j}}M_{L(\log L)^{\sum_{i\in\sigma'}\frac{1}{s_{i}}+\varepsilon}}(w_{j})(z)\int_{Q_{j}}\Phi_{\frac{1}{s}}\left(\frac{|f(x)|}{\lambda}\right)dx\\
 & \leq c_{k}\frac{1}{\varepsilon^{k}}\sum_{j}\inf_{z\in Q_{j}}M_{L(\log L)^{\frac{1}{s}+\varepsilon}}(w_{j})(z)\int_{Q_{j}}\Phi_{\frac{1}{s}}\left(\frac{|f(x)|}{\lambda}\right)dx\\
 & \leq c_{k}\frac{1}{\varepsilon^{k}}\sum_{j}\int_{Q_{j}}M_{L(\log L)^{\frac{1}{s}+\varepsilon}}(w_{j})(x)\Phi_{\frac{1}{s}}\left(\frac{|f(x)|}{\lambda}\right)dx\\
 & \leq c_{k}\frac{1}{\varepsilon^{k}}\int_{\mathbb{R}^{n}}M_{L(\log L)^{\frac{1}{s}+\varepsilon}}(w_{j})(x)\Phi_{\frac{1}{s}}\left(\frac{|f(x)|}{\lambda}\right)dx
\end{split}
\]
For $L_{32}$ arguing in the same way we have that 
\[
E_{2}\leq c_{k}\frac{1}{\varepsilon^{k}}\sum_{i=1}^{k-1}\sum_{\sigma\in C_{i}^{k}}\sum_{j}\inf_{z\in Q_{j}}M_{L(\log L)^{\frac{1}{s}+\varepsilon}}w_{j}(z)\int_{Q_{j}}\Phi_{\sum_{i\in\sigma'}\frac{1}{s_{i}}}\left(\frac{|f_{Q_{j}}|}{\lambda}\left(b(x)-b_{Q_{j}}\right)_{\sigma}\right)dx
\]
The same computation used to obtain \eqref{eq:1} yields
\[
\int_{Q_{j}}\Phi_{\sum_{i\in\sigma'}\frac{1}{s_{i}}}\left(\frac{|f_{Q_{j}}|}{\lambda}\left(b(x)-b_{Q_{j}}\right)_{\sigma}\right)dx\leq\int_{Q_{j}}\Phi_{\frac{1}{s}}\left(\frac{|f_{Q_{j}}|}{\lambda}\right)dx+ck|Q_{j}|.
\]
Now we see that using Jensen's inequality, 
\[
\begin{split} & \int_{Q_{j}}\Phi_{\frac{1}{s}}\left(\frac{|f_{Q_{j}}|}{\lambda}\right)dx\leq|Q_{j}|\Phi_{\frac{1}{s}}\left(\frac{|f|_{Q_{j}}}{\lambda}\right)\\
\text{} & \leq|Q_{j}|\frac{1}{|Q_{j}|}\int_{Q_{j}}\Phi_{\frac{1}{s}}\left(\frac{|f(x)|}{\lambda}\right)dx=\int_{Q_{j}}\Phi_{\frac{1}{s}}\left(\frac{|f(x)|}{\lambda}\right)dx.
\end{split}
\]
Hence 
\[
\int_{Q_{j}}\Phi_{\sum_{i\in\sigma'}\frac{1}{s_{i}}}\left(\frac{|f_{Q_{j}}|}{\lambda}\left(b(x)-b_{Q_{j}}\right)_{\sigma}\right)dx\leq\int_{Q_{j}}\Phi_{\frac{1}{s}}\left(\frac{|f(x)|}{\lambda}\right)dx+ck|Q_{j}|
\]
and we finish the estimate arguing as we did for $L_{31}$.
\end{proof}

\bibliographystyle{plainnat}
\bibliography{refs}

\end{document}